\documentclass[12pt]{article}
\usepackage[utf8]{inputenc}
\usepackage{url, color, epsfig, amsmath, amsthm, amssymb}
\usepackage{hyperref}
\usepackage{enumerate}
\usepackage{graphicx,wrapfig,lipsum}
\usepackage{svg}
\usepackage{thmtools}
\usepackage{subcaption}
\usepackage{float}
\usepackage{algorithm,algorithmic}
\usepackage{lmodern}
\usepackage{apxproof}
\usepackage{thmtools}
\usepackage{thm-restate}
\usepackage{svg}
\usepackage[T1]{fontenc}
\usepackage{hyperref}
\usepackage{cleveref}

\newcommand{\hide}[1]{}

\newcommand{\BB}{\mathcal{B}}
\newcommand{\GG}{\mathcal{G}}

\newcommand{\R}{\mathbf{r}}
\newcommand{\B}{\mathbf{b}}

\newcommand{\HH}{\mathcal{H}}
\newcommand{\KK}{\mathcal{K}}

\newcommand{\tw}{\textsc{tw}}

\newcommand{\arc}{\mathrm{arc}}

\newcommand{\RR}{\mathbb{R}}

\newcommand{\unique}{\textrm{unique}}
\newcommand{\poly}{\textrm{poly}}

\newtheorem{theorem}{Theorem}
\newtheorem{note}[theorem]{Note}

\newtheorem{lemma}[theorem]{Lemma}

\newtheorem{definition}[theorem]{Definition}

\newtheorem{proposition}[theorem]{Proposition}

\usepackage{tabularx,environ}
\makeatletter
\newcommand{\problemtitle}[1]{\gdef\@problemtitle{#1}}% Store problem title
\newcommand{\probleminput}[1]{\gdef\@probleminput{#1}}% Store problem input
\newcommand{\problemquestion}[1]{\gdef\@problemquestion{#1}}% Store problem question
\NewEnviron{problem}{
  \problemtitle{}\probleminput{}\problemquestion{}% Default input is empty
  \BODY% Parse input
  \par\addvspace{.5\baselineskip}
  \noindent
  \begin{tabularx}{\textwidth}{@{\hspace{\parindent}} l X c}
    \multicolumn{2}{@{\hspace{\parindent}}l}{\@problemtitle} \\% Title
    \textbf{Input:} & \@probleminput \\% Input
    \textbf{Question:} & \@problemquestion% Question
  \end{tabularx}
  \par\addvspace{.5\baselineskip}
}
\makeatother

\usepackage[colorinlistoftodos,prependcaption,textsize=tiny]{todonotes}

\newtheoremstyle{case}{}{}{}{}{}{:}{ }{}
\theoremstyle{case}

\title{Supports for Outerplanar and Bounded Treewidth Graphs\footnote{Parts of this work appeared in ``On Hypergraph Supports'' 
available at \url{https://arxiv.org/abs/2303.16515}, which has now been
split into two papers. The first paper is available at \url{https://arxiv.org/abs/2503.21287}, and this is the second paper in the series.}}

\author{Rajiv Raman\\ IIIT-Delhi, India. \\ rajiv@iiitd.ac.in\and Karamjeet Singh \\ IIIT-Delhi, India. \\ karamjeets@iiitd.ac.in}
\begin{document}

\maketitle
\begin{abstract}
We study the existence and construction of sparse supports for hypergraphs derived from subgraphs of a graph $G$. For a hypergraph 
$(X,\mathcal{H})$, a support $Q$ is a graph on $X$ s.t.
$Q[H]$, the graph induced on vertices in $H$
is connected for every $H\in\mathcal{H}$.

We consider \emph{primal}, \emph{dual}, and 
\emph{intersection} hypergraphs defined by subgraphs 
of a graph $G$ that are \emph{non-piercing}, 
(i.e., each subgraph is connected, 
their pairwise differences remain connected).

If $G$ is outerplanar, we show that the primal, dual and
intersection hypergraphs admit supports that are outerplanar.
For a bounded treewidth graph $G$, we show
that if the subgraphs are non-piercing, then there exist
supports for the primal and dual hypergraphs of treewidth
$O(2^{\tw(G)})$ and $O(2^{4\tw(G)})$ respectively, and a support of treewidth $2^{O(2^{\tw(G)})}$
for the intersection hypergraph. We also show that for
the primal and dual hypergraphs, the exponential blow-up
of treewidth is sometimes essential. 

All our results are algorithmic and yield polynomial time
algorithms (when the treewidth is bounded).
The existence and construction of sparse supports is a crucial step
in the design and analysis of
PTAS for several packing and covering problems.
\end{abstract}

\section{Introduction}\label{sec:int}
A hypergraph is defined by a set $X$ of \emph{vertices}, and a collection $\mathcal{H}$
of subsets of $X$ called \emph{hyperedges}.
A \emph{support} for $(X,\mathcal{H})$ is a graph $Q=(X,F)$ s.t. 
$Q[H]$ is connected for each $H\in\mathcal{H}$, where $Q[H]$ is the subgraph of
$Q$ induced by the vertices in $H$.
A complete graph on $X$ is clearly a support, but the existence of a support is non-trivial if we restrict the support to be a sparse graph.
For example, Johnson and Pollack~\cite{Johnson1987Hypergraph} showed that it is NP-complete to decide if a hypergraph admits a planar support.
Buchin et al. \cite{buchin2011planar} showed that it is NP-hard to decide if s hypergraph admits a support that is \emph{k-outerplanar} for any $k\ge 2$.

Voloshina and Feinberg introduced planar supports as a notion of planarity of hypergraphs~\cite{voloshina1984planarity}.
Since, then the existence and construction of planar supports has found applications in several domains such as visualization~\cite{bereg2015colored,bereg2011red,brandes2010blocks,brandes2012path,buchin2011planar,havet2022overlaying,hurtado2018colored} and network design~\cite{anceaume2006semantic,baldoni2007tera,baldoni2007efficient,chand2005semantic,hosoda2012approximability,korach2003clustering,onus2011minimum}.
The existence of sparse supports and similar questions have also found applicability for geometric problems~\cite{Mustafa17,BasuRoy2018,Cohen-AddadM15,krohn2014guarding,mustafa2010improved,RR18}.

In this paper, we consider hypergraphs defined on a \emph{host graph}. That is, we are given as input a graph $G=(V,E)$ along with a collection $\mathcal{H}$
of subgraphs of $G$, and our goal is to construct a support for a 
hypergraph defined on $G$. 
We focus on restrictions on $G$ and the subgraphs $\mathcal{H}$ 
that guarantee the existence of a \emph{sparse support}.

Our motivation comes from the fact that the existence of a support implies
a PTAS for a large class of packing and covering problems via a now well-established
technique. See for example,~\cite{DBLP:conf/walcom/AschnerKMY13,ChanH12, mustafa2010improved, RR18,raman2025supportsgraphsboundedgenus}. We now describe our general framework before
giving some potential applications. 

Let $G=(V,E)$ be a given graph and let 
$\B(V)\subseteq V$ be a set of \emph{terminals}.
Let $\R(V)$ denote the vertices in $V\setminus\B(V)$.
The set $\R(V)$ constitute the \emph{non-terminals}.
A family $\mathcal{H}$ of subgraphs of $G$
defines a hypergraph $(\B(V),\mathcal{H})$, 
whose vertex set is $\B(V)$ and
each $H\in\mathcal{H}$ defines a hyperedge 
consisting of the vertices in $\B(V)\cap V(H)$.
We call this hypergraph the \emph{primal hypergraph}.
Similarly, we define the \emph{dual hypergraph}
whose elements are the subgraphs in $\mathcal{H}$, and each $v\in\B(V)$ defines
a hyperedge consisting of all $H\in\mathcal{H}$ containing $v$.
We also consider a generalization of the primal and dual hypergraphs namely, the \emph{intersection hypergraphs}. 
Let $\mathcal{H}$ and $\mathcal{K}$ be two families of subgraphs of $G$.
The \emph{intersection hypergraph} defined by $\mathcal{H}$ and $\mathcal{K}$
is the hypergraph 
$
(\mathcal{H},\{\mathcal{H}_K\}_{K\in\mathcal{K}}),\mbox{ where } \mathcal{H}_K=\{H\in\mathcal{H}:V(H)\cap V(K)\ne\emptyset\}
$

The intersection hypergraphs are a common generalization of the primal and dual hypergraphs.
Indeed, setting $\mathcal{H}=\B(V)$, we obtain the primal hypergraphs, and setting
$\mathcal{K}=V$, we obtain the dual hypergraphs. Thus, any results that hold for intersection
hypergraphs also hold for primal and dual hypergraphs.

A support for a primal (resp. dual/intersection) hypergraph
is called a primal (resp. dual/intersection) support.
Formally, a \emph{primal support} is a graph $Q$ on $\B(V)$ such that for each $H\in\mathcal{H}$, the vertices $V(H)\cap\B(V)$ induce a connected subgraph of $Q$.
A \emph{dual support} is a graph $Q^*$ on $\HH$ such that for each $v\in\B(V)$, the elements in $\{H\in\HH:v\in H\}$ induce a connected subgraph of $Q^*$.
Finally, an \emph{intersection support} is a graph $\tilde{Q}$ on $\HH$ such that for each $K\in\KK$, the elements in $\HH_K$ induce a connected subgraph of $\tilde{Q}$.
In the dual setting, note that we can assume without loss of generality
that $\B(V)=V$.

See Figure 2 in \cite{raman2025supportsgraphsboundedgenus} for examples of primal, dual and intersection
hypergraphs arising from a given host graph, and their corresponding primal, dual, and intersection supports.

We design algorithms to construct primal, dual
and intersection supports for graphs of bounded treewidth.
We obtain similar results for a special case of graphs of 
treewidth 2, namely, the outerplanar graphs.

In the following, the class $\mathcal{G}$ consists either of the family of bounded treewidth graphs, or outerplanar graphs.

\begin{problem}
  \problemtitle{Primal Support}
  \probleminput{A graph $G\in\GG$ with $c:V(G)\to\{\R,\B\}$, and a collection $\mathcal{H}$ of connected subgraphs of $G$.}
  \problemquestion{Is there a support $Q$ on $\B(V(G))$ such that $Q\in\GG$, i.e., $Q$ has bounded treewidth and for each $H\in\mathcal{H}$,
  $H\cap\B(V(G))$ induces a connected subgraph in $Q$?}
\end{problem}

\begin{problem}
      \problemtitle{Dual Support}
  \probleminput{A graph $G\in\GG$ and a collection $\HH$ of connected subgraphs of $G$.}
  \problemquestion{Is there a dual support $Q^*\in\GG$ on $\mathcal{H}$, i.e., for each $v\in V(G)$,
  $\{H\in\mathcal{H}: v\in H\}$ induces a connected subgraph of $Q^*$?}
\end{problem}

\begin{problem}
      \problemtitle{Intersection Support}
  \probleminput{A graph $G\in\GG$ and two collections $\mathcal{H}$ and $\mathcal{K}$ of connected subgraphs of $G$.}
  \problemquestion{Is there an intersection support $\tilde{Q}\in\GG$ on $\mathcal{H}$, 
  i.e., for each $K\in\mathcal{K}$, the set $\{H\in\mathcal{H}: V(H)\cap V(K)\neq\emptyset\}$ induces a connected subgraph of $\tilde{Q}$?}
\end{problem}

As we show below, for the problem defined above,
the answer is negative even if $\mathcal{G}$ is the family of trees.
However, if we restrict the family of subgraphs to be \emph{non-piercing},
then the answer to the question above is affirmative.
A family $\HH$ of connected subgraphs of a graph $G$ is said to be non-piercing, if for each $H,H'\in\HH$, both $V(H)\setminus V(H')$ and $V(H')\setminus V(H)$ induce connected subgraphs of $G$.

The motivation to consider hypergraphs defined by non-piercing subgraphs of a graph comes from geometry.
Raman and Ray \cite{RR18} showed that the intersection hypergraph of \emph{non-piercing regions}\footnote{A family $\mathcal{R}$ of connected regions in $\mathbb{R}^2$ is \emph{non-piercing} if for each $R,R'\in\mathcal{R}$, both $R\setminus R'$ and $R'\setminus R$ are connected regions. They generalize disks, pseudodisks, and halfspaces in $\RR^2$} in the plane admits a planar support.
It was shown in \cite{raman2025supportsgraphsboundedgenus} that non-piercing regions in $\mathbb{R}^2$ correspond to non-piercing subgraphs of a planar graph.
The work was extended to surfaces of higher genus where the subgraphs satisfy a condition of being \emph{cross-free}.
The authors showed that the intersection hypergraph defined by cross-free subgraphs of a bounded genus graph admits a support of bounded genus.
It was observed that non-piercing is an insufficient condition to obtain a dual support of bounded genus.
In this paper, we show that the non-piercing condition is sufficient to obtain an intersection support of bounded treewidth.
In order to construct an intersection support, we require the construction of
primal and dual supports.

Since the graph classes we consider here are simple, one may
wonder if the non-piercing condition is necessary
to obtain sparse supports. The following examples show that
this is indeed the case. 
For the primal hypergraph,
consider a star $K_{1,n}$
with leaves $v_1,\ldots, v_n$, and $v_0$, the central vertex.
We color the leaves $\B$ and
the central vertex $\R$. The subgraphs in $\mathcal{H}$ consist
of all pairs of leaves along with the central vertex $v_0$. Thus, $\mathcal{H}=\{H_{ij}: H_{ij}=\{v_i,v_0, v_j\}\}$. It is easy to see that the primal support in this case
is $K_n$, the complete graph on $n$ vertices. 
For the dual hypergraph, consider a star $K_{1,\binom{n}{2}}$
with central vertex $v_0$.
Each leaf $v_j$ is labeled by a unique pair $\{x_j,y_j\}$
of $\{1,\ldots, n\}$. 
There are $n$ subgraphs $H_1,H_2,\ldots,H_n$, where $H_i=\{v_0\}\cup\{v_j: i\in\{x_j,y_j\}\}$.
It is easy to check that the dual support in this case is also $K_n$,
as there are exactly two subgraphs containing each leaf.

Below, we give a few potential applications where our techniques
can apply. Since our focus is on the existence and construction of
supports, we do not elaborate more on the applications as the
algorithmic results follow via standard techniques once there
is a support graph.

Consider a graph $G$. Note that a set of cliques in $G$ is
a natural non-piercing family of subgraphs of $G$. Given two
sets of cliques of $G$, a set of \emph{red cliques} and a
set of \emph{blue cliques}, consider the problem of 
choosing the smallest number of red cliques s.t. each blue clique shares a vertex
with at least one of the chosen red cliques.
If $G$ has bounded treewidth, since we show that there exists an intersection support
of bounded treewidth, it implies that this problem admits a PTAS via a standard
\emph{local-search} framework described in~\cite{DBLP:conf/walcom/AschnerKMY13,ChanH12, mustafa2010improved}. Note that in this problem, the set of 
red/blue cliques can be specified arbitrarily.

As another potential example,  
let $\mathcal{C}$ be a set of cliques of a graph $G$. 
Then, the problem of choosing the largest subset $S$ of
$\B(V)$ so that each clique in $\mathcal{C}$ contains
at most one vertex of $S$ admits a sub-exponential time
algorithm for graphs with treewidth at most $c\log n$, 
for some sufficiently small $c<1$. 
The sub-exponential time algorithm is as follows: 
we show that a support graph $Q$
can be computed in polynomial time, and $Q$ has treewidth
at most $2^{O(\tw(G))}=n^{\delta}$ for some $\delta<1$.
We enumerate over the subsets of vertices in a 
balanced separator of $Q$ and then recursively find
such subgraphs in the two sub-problems obtained on removing
the separator while ensuring that the choices made
in the sub-problems are compatible with the choices made
further up the recursion tree. 
Since $Q$ is a support, each clique that
contains vertices on either side of the separator must necessarily
contain vertices in the separator. Since the separator is balanced,
the recursion tree has depth $O(\log n)$ and 
since the number of vertices of $Q$ is at most $n$, and
the separator has sublinear size, we obtain a
sub-exponential time algorithm.

\section{Related Work}\label{sec:relwrk}
The notion of planarity for hypergraphs, unlike that for graphs
is not uniquely defined. This notion was first studied by
Zykov~\cite{zykov}, but that was very restrictive.
Voloshina and Feinberg~\cite{voloshina1984planarity}
defined a notion of planarity of hypergraphs that is equivalent to
the existence of a support graph that is planar, i.e., a \emph{planar support}.
Johnson and Pollak~\cite{Johnson1987Hypergraph} showed that deciding if a hypergraph
admits a planar support is NP-hard.
Buchin et al. \cite{buchin2011planar} strengthened their result
showing that it is NP-hard even to decide if a hypergraph admits a support
that is  $k$-outerplanar, for $k\ge 2$.
However, one can decide in polynomial time if a hypergraph has a support that is a path, a cycle, a \emph{cactus}\footnote{A connected graph is called cactus if each of its edges participates in at most one cycle.}, or a tree with bounded maximum degree~\cite{beeri,BOOTH1976335,brandes2010blocks,buchin2011planar,tarjan1984simple}.
Recently, Chudnovsky et al. \cite{chudnovsky2025vertex} gave characterizations of hypergraphs having a tree support.

Besides, other notions of sparsity have also been studied.
Du~\cite{du1986optimization} proved that it is NP-hard to construct a support with the minimum number of edges.

As stated in the introduction, the existence of a sparse support for an appropriate
hypergraph implies a PTAS for several packing and covering problems via a
local-search framework. See~\cite{BasuRoy2018,ChanH12,mustafa2010improved} for various packing and covering problems in the planar setting including the Independent Set, Dominating Set, Set Cover problems. 
Raman and Ray~\cite{RR18} showed the existence of a planar support for the intersection hypergraph defined by non-piercing regions in the plane. This result
was extended by Raman and Singh~\cite{raman2025supportsgraphsboundedgenus}, who showed that
hypergraphs defined by
cross-free subgraphs of a graph of genus $g$ also admit a supports of genus at most $g$.
Using this, the authors showed that if the hypergraph is defined by a \emph{restrictive} collection of non-piercing regions on a surface of genus $g$, then there is a support of genus $g$.
These results imply a PTAS for all the problems mentioned above.

\section{Preliminaries}\label{sec:prelim}
Let $G=(V,E)$ be a graph and let $\mathcal{H}$ be a collection of 
connected subgraphs of $G$. For $v\in V$, 
let $\mathcal{H}_v=\{H\in\mathcal{H}:v\in V(H)\}$.
Similarly, for an edge $e\in E(G)$, 
we write $\mathcal{H}_{e}=\{H\in\mathcal{H}:e\in E(H)\}$.
We colloquially use $H$ to denote the vertices $V(H)$ and
for two subgraphs $X,Y$ of a graph $G$, we write $X\cap Y$ to mean $V(X)\cap V(Y)$, $X\setminus Y$ to mean $V(X)\setminus V(Y)$, and $X\subseteq Y$ to mean $V(X)\subseteq V(Y)$.

Let $\B(V)$ and $\R(V)$ be a partition of $V(G)$ into a set 
of \emph{terminals} and a set of \emph{non-terminals}, respectively.
Formally, let $c:V(G)\to\{\R,\B\}$ be a 2-coloring of $V(G)$ (not necessarily proper).
We call $\R(V)=c^{-1}(\R)$ the set of \emph{red} vertices and $\B(V)=c^{-1}(\B)$ the set of \emph{blue} vertices.
We use $\B(H)$ to denote $V(H)\cap\B(V)$, and $\R(H)$ to denote $V(H)\cap\R(V)$.

\begin{definition}[Non-piercing subgraphs]\label{def:nonpiercing} A family $\mathcal{H}$ of connected subgraphs of a graph $G$ is called non-piercing if the induced subgraph of $G$ on the vertices $V(H)\setminus V(H')$ is connected $\forall\;H,H'\in\mathcal{H}$.
\end{definition}

\begin{definition}[Outerplanar graph]
A graph $G$ is called outerplanar if there is an embedding of $G$ in the plane such that all its vertices lie on the exterior face.
\end{definition}

Next, we define the notion of \emph{tree decomposition}.
Throughout the paper, we use the term \emph{node} to refer to the elements of $V(T)$ for a tree $T$, and we use \emph{vertices} to refer to the elements of $V(G)$ for a graph $G$.
We use letters $x,y,z$ for the nodes of $T$, and $u,v,w$ for the vertices of $G$.

\begin{definition}[Tree decomposition] Given a graph $G=(V,E)$, a tree decomposition of $G$ is a pair $(T,\mathcal{B})$, 
where $T$ is a tree and $\mathcal{B}$ is a collection of \emph{bags} - subgraphs of $G$ indexed by the nodes of $T$, that satisfies the following properties:
\begin{enumerate}
    \item For each $v\in V(G)$, the set of bags of $T$ containing $v$ induces a sub-tree of $T$.
    \item For every edge $\{u,v\}$ in $G$, there is a bag $B\in\mathcal{B}$ such that $u,v\in B$.
\end{enumerate}
\end{definition}
\begin{definition}[Treewidth]
The width of a tree decomposition $(T,\mathcal{B})$ is defined to be   $\max_{x\in V(T)} |B_x|-1$.
The treewidth of a graph $G$ is the minimum width over all the tree decompositions of $G$, and is denoted $\tw(G)$. 
\end{definition}

We use two well-known properties of tree decompositions. First, that we can assume that
$T$ is a rooted binary tree. That is, we can modify a given tree decomposition to one where
$T$ is a rooted binary tree without increasing the treewidth. 
If $T$ is rooted, then we use $T_x$ to denote the subtree rooted at node $x$.
The second is
that for any edge $e=\{x,y\}$ of $T$, the set $B_x\cap B_y$ is a \emph{separator}
in $G$. That is, the induced subgraph $G'=G\setminus\{B_x\cap B_y\}$ of $G$ is disconnected.
The set $B_x\cap B_y$ is called
the \emph{adhesion set} corresponding to edge $\{x,y\}$ of $T$. 
We refer the reader to Chapter 12 in~\cite{diestel2005graph} for more details on the properties of
treewidth and tree decompositions. 

For a graph $G$ and collections $\HH$ and $\KK$ of subgraphs of $G$, we use the pair $(G,\HH)$ to refer to the primal or dual hypergraph, and triple $(G,\HH,\KK)$ for the intersection hypergraph.
We call the tuple $(G,\mathcal{H})$ \emph{a graph system} and the tuple $(G,\mathcal{H},\mathcal{K})$ \emph{an intersection system}.
If $G\in\mathcal{G}$, and $\mathcal{H,K}$ satisfy a property $\Pi$, then we say that
$(G,\mathcal{H})$ is a $\mathcal{G}$ $\Pi$ graph system, or that
$(G,\mathcal{H},\mathcal{K})$ is a $\mathcal{G}$ $\Pi$ intersection system.
For example, 
if $G$ has bounded treewidth, and $\mathcal{H}$ and $\mathcal{K}$ are connected
subgraphs of $G$, then $(G,\mathcal{H},\mathcal{K})$ is a bounded treewidth connected intersection system.
Similarly, if $G$ is outerplanar and $\mathcal{H}$ is non-piercing, then
$(G,\mathcal{H})$ is an outerplanar non-piercing system.

\begin{note}
Since two subgraphs are said to intersect if they share a vertex, a hypergraph defined by subgraphs of a graph $G$ remains the same if the subgraphs are replaced by the corresponding induced subgraphs.
Moreover, if $\HH$ is a collection of non-piercing subgraphs, then it remains non-piercing if the subgraphs are replaced by the corresponding induced subgraphs.
Therefore, we assume throughout that the term \emph{subgraph} refers to an
\emph{induced subgraph} of $G$.
\end{note}

Also note that adding edges to $G$ leaves the subgraphs in $\mathcal{H}$ or $\mathcal{K}$ non-piercing. Thus, we can 
assume in the outerplanar setting that $G$ is a maximal outerplanar graph, and in the setting where $G$ has bounded
treewidth, we can assume that each bag in a tree decomposition of the graph induces a complete graph, as this does 
not increase the treewidth (see Section~\ref{sec:treewidth} for more details).

Let $\mathcal{H}$ be a collection of subgraphs of $G$. For two subgraphs $H, H'\in\mathcal{H}$ 
if $H\subseteq H'$, we say that $H$ is \emph{contained in} $H'$. If there are no such subgraphs $H, H'$, we say that
$\mathcal{H}$ is \emph{containment-free}. The containments in $\mathcal{H}$ yields a natural partial order $\prec_N$ on $\mathcal{H}$,
where $H\prec_N H' \Leftrightarrow V(H)\subseteq V(H')$.
Let $\mathcal{H}^*\subseteq\mathcal{H}$ be the containment-free subfamily of $\mathcal{H}$ consisting of the maximal elements of this partial order.
For an intersection support (and hence for a dual support), 
the following result 
shows that it is sufficient to 
assume that $\mathcal{H}$ is containment-free.

\begin{lemma}
\label{lem:contfree}
Let $(G,\mathcal{H,K})$ be a graph system and  $\mathcal{H}^*\subseteq\mathcal{H}$ be the containment-free
subfamily of $\mathcal{H}$ consisting of all the maximal elements of $(\mathcal{H},\prec_N)$.
Then,
\begin{enumerate}
    \item If $(G,\mathcal{H^*,K})$ has an outerplanar support, then $(G,\mathcal{H,K})$ also has an outerplanar support.
    \item If $(G,\mathcal{H^*,K})$ has a support of treewidth $t$, then $(G,\mathcal{H,K})$ also has a support of treewidth $t$.
\end{enumerate}

\end{lemma}
\begin{proof}
By definition, any subgraph $H\in\mathcal{H}\setminus\mathcal{H}^*$ has a successor $H'\in\mathcal{H}^*$ in the poset $(\mathcal{H},\preceq_N)$ i.e., $H\subseteq H'$ for some $H'\in\mathcal{H}^*$.
Let $\tilde{Q}'$ be a support for $(G,\mathcal{H}^*,\mathcal{K})$, and $\tilde{Q}$ be a graph obtained from $\tilde{Q}'$ by adding a vertex for each $H\in\mathcal{H}\setminus\mathcal{H}^*$ and making it adjacent to one of its successors $H'\in\mathcal{H}^*$ (breaking ties arbitrarily).
Since $\tilde{Q}$ is obtained form $\tilde{Q}'$ by adding only degree one vertices, $\tw(\tilde{Q})=\tw(\tilde{Q}')$. Moreover, $\tilde{Q}$ remains outerplanar if $\tilde{Q}'$ is outerplanar.

We now show that $\tilde{Q}$ is an intersection support for $(G,\mathcal{H,K})$.
Let $K\in\mathcal{K}$ be arbitrary.
We claim that $\tilde{Q}[\mathcal{H}_K]$ induces a connected subgraph of $\tilde{Q}$.
Since $\tilde{Q}'$ is an intersection support for $(G,\mathcal{H}^*,\mathcal{K})$, the induced subgraph $\tilde{Q}'[\mathcal{H}^*_K]$ of $\tilde{Q}'$ is connected.
By definition, $\mathcal{H}^*\subseteq\mathcal{H}$ consists of precisely the maximal elements of $(\mathcal{H},\preceq_N)$.
Therefore, for any $H\in\mathcal{H}_K\setminus\mathcal{H}^*_K$, $\exists H'\in\mathcal{H}^*$ such that $H\subseteq H'$ and $\{H,H'\}$ is an edge in $\tilde{Q}$.
But $H'$ intersects $K$ since $H$ interacts $K$.
It follows that $H'\in\mathcal{H}^*_K$ and hence, the induced subgraph $\tilde{Q}[\mathcal{H}_K]$ of $\tilde{Q}$ is connected.
This completes the proof.
\end{proof}

\section{Contributions}\label{sec:contributions}
We give polynomial time algorithms to construct primal, dual, and intersection supports for hypergraphs defined by non-piercing subgraphs of a host graph that is outerplanar or has bounded treewidth.
For a non-piercing graph system of treewidth $t$, 
we show that the supports for the primal and dual hypergraphs have treewidth $O(2^t)$ and $O(2^{4t})$ respectively, and
the intersection support has treewidth $2^{O(2^{t})}$.

\begin{restatable}{theorem}{primalsupporttw}
\label{thm:primalsupporttw}
Let $(G,\mathcal{H})$ be a non-piercing graph system of 
treewidth $t$ with $c:V(G)\to\{\R,\B\}$.
There
is a primal support $Q$ of treewidth  at most $2^{t+2}+t$.
\end{restatable}

\begin{restatable}{theorem}{dualsupporttw}
\label{thm:dualsupporttw}
Let $(G,\mathcal{H})$ be a non-piercing graph system of treewidth $t$. There is a dual support $Q^*$
of treewidth at most $2^{4(t+1)}$.
\end{restatable}

\begin{restatable}{theorem}{intsupporttw}
\label{thm:intsupporttw}
Let $(G,\mathcal{H,K})$ be a non-piercing intersection system of treewidth $t$.
Then, there is an intersection support $\tilde{Q}$ 
of treewidth at most $2^{2^{t+4}+4(t+1)}$. 
\end{restatable}

For the primal and dual settings, we show that the exponential blow-up in the treewidth of the support is sometimes
necessary.

\begin{restatable}{theorem}{primallb}
\label{thm:primallb}
For any $m\in\mathbb{N}$,
there is a non-piercing graph system $(G,\mathcal{H})$ 
with $c:V(G)\to\{\R,\B\}$ s.t. $\tw(G)\le m$, but
$\tw(Q)\ge \frac{2^{\lfloor m/2\rfloor}}{\sqrt{m}}$ for any primal support $Q$.
\end{restatable}

\begin{restatable}{theorem}{duallb}
\label{thm:dualtwlb}
For any $m\in\mathbb{N}$,
there is a non-piercing graph system $(G,\mathcal{H})$
s.t. $\tw(G)\le m$ and for any dual support $Q^*$,
$\tw(Q^*)\ge \frac{2^{\lfloor m/2\rfloor}}{\sqrt{m}}$.   
\end{restatable}

In case the host graph $G$ is outerplanar, we show that the non-piercing intersection system $(G,\HH,\KK)$ admits an outerplanar support.

\begin{restatable}{theorem}{outerplanar}
\label{thm:np-outintsupport}
If $(G,\mathcal{H,K})$ is an outerplanar non-piercing system, then there is an intersection support for $(G,\mathcal{H,K})$ which is an outerplanar graph.
\end{restatable}

Theorem \ref{thm:np-outintsupport} also implies the existence of primal and dual supports for an outerplanar
non-piercing graph system.
The authors in \cite{buchin2011planar} left it open, the problem of deciding if a hypergraph admits an outerplanar support.
While we don't answer their question, we believe that the tools we develop for constructing outerplanar support can eventually help in resolving this decision problem

The rest of the paper is organized as follows. In Section \ref{sec:treewidth}, we present the construction of support for bounded treewidth graphs systems, and in Section \ref{sec:lowerbounds}, we show lower bounds for the treewidth of primal and dual supports.
In Section \ref{sec:outerplanar},
we present results on outerplanar graphs.
We conclude in Section \ref{sec:conclusion}.

\section{Support of Bounded Treewidth}
\label{sec:treewidth}
In this section, we show that if $(G,\mathcal{H,K})$ is a non-piercing system of treewidth $t$, then there exists an intersection support
of treewidth $2^{O(2^{t})}$. In order to show this, we 
construct a primal and dual supports of treewidth $O(2^{t})$ and $O(2^{4t})$ respectively, 
and use these to construct an intersection support.
Further, our construction is FPT in $\tw(G)$.
First, we collect some tools required for the construction of supports.

For a graph system $(G,\mathcal{H})$, let $(T,\mathcal{B})$ 
denote a tree decomposition of $G$ of width $t$. Let $CC(G)$ denote the
\emph{chordal completion} of $G$, i.e., we add edges so that
all vertices in a bag are pairwise adjacent. 
Clearly, a chordal completion does not increase the width of the tree decomposition.
Thus if $(T,\BB)$ is of width $\tw(G)$, then $\tw(CC(G))=\tw(G)$.
Moreover, if $\HH$ is a collection of non-piercing subgraphs of $G$, they remain non-piercing when induced on $CC(G)$ since the underlying hypergraph does not change.
Further, it is straightforward that computing a support w.r.t. $CC(G)$ is equivalent
to computing a support w.r.t. $G$.
Hence, we assume 
throughout this chapter that, $G$ is a \emph{chordal graph}\footnote{A graph is chordal, if there is no induced cycle of length at least 4.}.

We use the following notations: We assume throughout this section that $T$ be a binary tree rooted at node $\rho$. For any node $x$ of $T$, we use $T_x$ to denote the sub-tree rooted at $x$.
Let $\BB_x$ denote the set of bags at the nodes in $T_x$.
The subgraph $G_x=\cup_{z\in T_x}B_z$, of $G$ is the graph induced on the vertices corresponding
to the bags in $\BB_x$.
Note that $(T_x,\BB_x)$ induces a tree decomposition of $G_x$.
We use $G_{-x}$ to denote the subgraph induced on $V(G)\setminus V(G_x)$.

Let $\HH$ be a collection of subgraphs of $G$.
For $H\in\mathcal{H}$, we let $H|_x=G_x\cap H$ 
and $\mathcal{H}|_x=\{H|_x: H\in\mathcal{H}\}$.
Similarly, $H|_{-x}=H\cap G_{-x}$ and $\mathcal{H}|_{-x}=\{H|_{-x}: H\in\mathcal{H}\}$.
Then, $(G_x,\mathcal{H}|_x)$ and $(G_{-x},\HH_{-x})$ denote
the two induced graph systems on $G_x$ and $G_{-x}$ respectively.

For an edge $\{x,y\}$ of $T$, we use $A_{xy}$ to denote the adhesion set $B_x\cap B_y$.
For $A\subseteq V(G)$, let
$\mathcal{H}_A=\{H\in\mathcal{H}: A\cap H\neq\emptyset\}$, and
for $S\subseteq A$,
$\mathcal{H}^=_S=\{H\in\mathcal{H}_A: H\cap A = S\}$.

The next two lemmas follow
from the fact that the subgraphs considered are non-piercing.
For two sets $A$ and $B$ on the same ground set, we say that $A$ and $B$ \emph{properly intersect} 
if $A\setminus B\neq\emptyset$ and $B\setminus A\neq\emptyset$.

\begin{lemma}
\label{lem:npsep}
Let $(G,\mathcal{H})$ be a non-piercing graph system with
a tree decomposition $(T,\mathcal{B})$ of $G$.
Let $e=\{x,y\}\in E(T)$ be an edge in $T$ where $x$ is a child of $y$.
Then for $H, H'\in\mathcal{H}_{A_{xy}}$, the following holds:
$(i)$ If $A_{xy}\cap H=A_{xy}\cap H'$ and $H|_x\subset H'|_x$, then
$H'|_{-x}\subseteq H|_{-x}$. $(ii)$ If $A_{xy}\cap H=A_{xy}\cap H'$, 
and $H|_x$ and $H'|_x$ properly intersect,
then $H|_{-x}=H'|_{-x}$, and $(iii)$ If $H\cap A_{xy}\subset H'\cap A_{xy}$, and $H|_x$ and $H'|_x$ properly intersect, then $H|_{-x}\subseteq H'|_{-x}$.
\end{lemma}

\begin{proof}
The proofs follow immediately from the fact that $\mathcal{H}$
are non-piercing and the adhesion set corresponding to 
any edge in $E(T)$ is a separator in $G$.
\end{proof}

\begin{lemma}
\label{lem:sepnp}
Let $(G,\mathcal{H})$ be a non-piercing graph system with
a tree decomposition $(T,\mathcal{B})$ of $G$ 
For any node $x$ of $G$, the graph system
$(G_x,\HH_x)$ is non-piercing and $(T_x,\BB_x)$
is a tree decomposition of $G_x$.
\end{lemma}
\begin{proof}
Since $G$ is chordal, each bag induces a complete subgraph of $G$. Hence, each
$H_x\in\mathcal{H}_x$ is connected.
If $x$ is not a root of $T$, let $y$ be the parent of $x$.
Then, for $H_x, H'_x\in\mathcal{H}_x$, the subgraph
$H_x\setminus H'_x$ of $G_x$ is connected since $H\setminus H'$ is connected
and $A_{xy}$ is a separator in $G$.
Therefore, $(G_x,\HH_x)$ is a non-piercing graph system.
Finally, that $(T_x,\mathcal{B}_x)$ is a 
tree decomposition, is straightforward.
\end{proof}

\subsection{Primal Support}
\label{sec:primal}
In this section, we show that a bounded treewidth non-piercing
system $(G,\mathcal{H})$ with $c:V(G)\to\{\R,\B\}$ admits
a primal support $Q$ s.t. $\tw(Q)\le 2^{\tw(G)+2}+\tw(G)$.
The proof is algorithmic and yields a polynomial time algorithm if $\tw(G)$ is bounded, i.e., an FPT-algorithm parameterized by
$\tw(G)$.

An \emph{easy tree decomposition} of $G$
is a tree decomposition s.t. 
for each adhesion set $A$ of the tree decomposition and
each subgraph $H$ intersecting $A$, it does so at a blue vertex, i.e.,
for any adhesion set $A$, if $H\cap A\neq\emptyset$, then $H\cap\B(A)\neq\emptyset$.
If $G$ has an easy tree decomposition of treewidth $t$, then it is straightforward to construct a primal support $Q$
s.t. $\tw(Q)\le t$. Moreover, in this case, 
we only require the subgraphs
in $\mathcal{H}$ to be connected.

\begin{lemma}
\label{lem:easytd}
Let $(G,\mathcal{H})$ be a graph system with $c:V(G)\to\{\R,\B\}$
s.t. each $H\in\mathcal{H}$ is a connected subgraph
of $G$.
If there is an easy tree decomposition $(T,\mathcal{B})$ of
width $t$, then there is a primal support $Q$ 
on $\B(V)$ of treewidth at most $t$.
\end{lemma}
\begin{proof}
Let $(T,\mathcal{B}')$ be the tree decomposition on $\B(V)$
derived from $(T,\mathcal{B})$, 
where $B'=B\cap\B(B)$ is the bag in $\mathcal{B}'$ corresponding to
the bag $B\in\mathcal{B}$. 
For each $B'\in\mathcal{B}'$,
the vertices of $B'$ induce a clique since we work with the
chordal completion of $G$.
Then, $(T,\mathcal{B}')$ is the
tree decomposition of a graph $Q$.  

To show that $Q$ is a support, consider an $H\in\mathcal{H}$.
Since $H$ is a connected subgraph in $G$, 
the vertices in $V(H)$ lie in bags corresponding to 
a connected sub-tree of $T$.
Further, as $(T,\mathcal{B})$ is an 
easy tree decomposition, $\B(V(H))$ lie in bags
of $\mathcal{B}'$ corresponding to a connected sub-tree of $T$.
Since each $B'\in\mathcal{B}'$ induces a complete subgraph
of $Q$, it implies $V(H)$ induces a connected subgraph of $Q$.
\end{proof}

In the proof below, we use the following
notation:
Let $\rho$ be the parent of $x$ in the tree decomposition $(T,\BB)$.
For an adhesion set $A_{x\rho}$ and $S\subseteq A_{x\rho}$, 
for notational convenience,
we assume that each $H\in\mathcal{H}^=_S$ is s.t. 
$H|_x\cap\B(G_x)\neq\emptyset$ and $H|_{-x}\cap\B(G_{-x})\neq\emptyset$.
Let
$\mathcal{M}_S\subseteq\mathcal{H}^=_S$ denote
the set of subgraphs $H$ such that $H_x$ is minimal in $G_x$ in the containment order $\preceq$, i.e., for $H,H'\in\HH^=_S$, 
$H|_x\preceq H'|_x \Leftrightarrow H|_x\subseteq H'|_x$.
We use $(G_x,\preceq)$ to denote this containment order.
We use minimal elements in $\mathcal{M}_S$ to construct an easy tree decomposition at the cost of an increase in the width of the tree decomposition.

\begin{lemma}
\label{lem:makeeasy}
Let $(G,\mathcal{H})$ be a non-piercing graph system with 
$c:V(G)\to\{\R,\B\}$. A tree decomposition $(T,\mathcal{B})$ of 
$G$ of width $t$ can be transformed into an easy tree decomposition
$(T,\mathcal{B}')$ of width at most $2^{t+2} + t$.
\end{lemma}
\begin{proof}
If $(T,\mathcal{B})$ is an easy tree decomposition, we are done. 
Otherwise, we modify $(T,\mathcal{B})$ to an easy tree decomposition
$(T,\mathcal{B}')$ 
by adding additional blue vertices to the bags in $\mathcal{B}$.
We assume without loss of generality that $(T,\mathcal{B})$ is a
binary tree rooted at a node $\rho$.

We prove by induction on the height of $T$. If $T$ has height 0,
then $T$ consists of a single node $\rho$ 
and hence for each $H\in\mathcal{H}$, $H\cap\B(B_\rho)\neq\emptyset$.
Otherwise, let $x$ and $y$ be children
of $\rho$. By Lemma~\ref{lem:sepnp}, the graph system
$(G_x,\mathcal{H}_x)$ is non-piercing
and $(T_x,\mathcal{B}_x)$ is a tree decompositions of $G_x$. By the
inductive hypothesis, there is an easy tree decomposition
$(T_x,\mathcal{B}'_x)$ of
width at most $2^{t+2}+t$. 
Analogously, there is an easy tree decomposition $(T_y,\mathcal{B}'_y)$
of width at most $2^{t+2}+t$ for $(G_y,\mathcal{H}_y)$.

For each $S\subseteq \R(A_{x\rho})$, consider an $H\in\mathcal{M}_S$. Since $H\cap B'_x\neq\emptyset$ and
$(T,\mathcal{B}'_x)$ is an easy tree decomposition, $H\cap \B(B'_x)\neq\emptyset$. Choose a $b\in H\cap\B(B'_x)$ and add it
to $B_\rho$. The tree decomposition remains a valid tree decomposition
as the bags containing $b$ correspond to a connected
subset of nodes of $T$.
Similarly, for each $S\subseteq \R(A_{y\rho})$,
choose an $H'\in \mathcal{M}_S$ and a vertex
$b'\in H'\cap\B(B'_y)$ and add $b'$ to $B_\rho$. 
Let $B'_\rho$ denote the
bag at $\rho$ after all subsets of $A_{x\rho}$ and $A_{y\rho}$
have been processed. 
Since $|A_{x\rho}|,|A_{y\rho}|\le t+1$ and $|B_{\rho}|\le t+1$, we have $|B'_\rho|\le 2\cdot 2^{t+1}+(t+1)$.
Therefore, width of the tree decomposition $(T,\BB')$ is at most $2^{t+2}+t$.

We claim that $(T,\mathcal{B}')$ is an easy tree decomposition. 
Suppose not. Let $H\in\mathcal{H}$ s.t. $H\cap B_\rho\ne\emptyset$ and 
$H\cap \B(B'_\rho)=\emptyset$.
Then, $H\cap A_{x\rho}\neq\emptyset\neq H\cap A_{y\rho}$ by assumption.
Let $H\cap A_{x\rho}=S$ and $H\cap A_{y\rho}=S'$.
Let $H'$ be the minimal
subgraph in $\mathcal{M}_S$ whose blue vertex was added to $B'_\rho$
and let $H''$ be the minimal subgraph in $\mathcal{M}_{S'}$ whose
blue vertex was added to $B'_\rho$. 
Then, $H$ and $H'$ are incomparable in $(G_x,\preceq)$.
By part $(ii)$ of Lemma~\ref{lem:npsep}, this implies $H$ and $H'$ are identical in
$G_{-x}$. 
Similarly,
$H$ and $H''$ are incomparable in $(G_y,\preceq)$ and by 
Lemma~\ref{lem:npsep} $H$ and $H''$ are identical in
$G_{-y}$. 
It follows that $H'|_x$ and $H''|_x$ are incomparable in $G_x$, and $H'|_y$ and $H''|_y$ are incomparable in $G_y$.
Also, note that $H''\cap A_{x\rho}=S=H'\cap A_{x\rho}$.
But this implies $H'\setminus H''$ has at least two 
connected components - one in $G_x$ and one in $G_y$ contradicting
the assumption that the subgraphs in $\HH$ are non-piercing.
Hence, $(T,\BB')$ is an easy tree decomposition.
\end{proof}

\primalsupporttw*
\begin{proof}
Let $(T,\mathcal{B})$ be a tree decomposition of $G$ of width $t$ where $T$ is a binary rooted tree.
If $(T,\mathcal{B})$ is an easy tree decomposition, we are done by Lemma~\ref{lem:easytd}. Otherwise,
we transform it into an easy tree decomposition of width at most $2^{t+2} + t$
by Lemma~\ref{lem:makeeasy}. We then obtain a primal support $Q$ of the same treewidth by
Lemma~\ref{lem:easytd}.
\end{proof}

\subsection{Dual Support}
\label{sec:dual}
We now construct a dual support.
For a non-piercing system $(G,\HH)$ of treewidth $t$, we construct a dual support of treewidth $O(2^{4t})$.
By Lemma~\ref{lem:contfree}, we assume that there are \emph{no containments}, i.e., there are no two subgraphs $H, H'\in\mathcal{H}$ such that  $H\subseteq H'$. 
As in the primal setting, we first show how we can obtain a 
dual support for a simple case. A support for the general
case is obtained by reducing it to the simple case.
In the following, since we construct a graph on $\mathcal{H}$, we abuse notation
and also use $H$ to denote the vertex in the dual support corresponding to $H\in\mathcal{H}$.

For a graph system $(G,\mathcal{H})$, a tree decomposition
$(T,\mathcal{B})$ of $G$ is said to be \emph{k-sparse} with
respect to $\mathcal{H}$ if for each bag
$B\in \mathcal{B}$, at most $k$ subgraphs in $\HH$ intersect $B$.
If $(G,\mathcal{H})$ admits such a tree decomposition, 
we say that $(G,\mathcal{H})$ is $k$-sparse with respect to $\HH$. 
Note that we do not restrict $\mathcal{H}$ to be non-piercing for $(G,\HH)$ to be $k$-sparse.

For a connected graph system $(G,\mathcal{H})$
with a tree decomposition $(T,\mathcal{B})$ of width $t$ that is $k$-sparse 
with respect to $\mathcal{H}$,
a dual support can be computed in time 
$O(k^2 t|V(G)||\mathcal{H}|)$.
The algorithm, {\bf $k$-SDS} that achieves this is as follows:
We construct a tree decomposition $(T,\BB')$ of a graph $Q^*$.
For each bag $B\in\mathcal{B}$, we put a vertex $H$ in the corresponding bag $B'\in\BB'$ if $H\cap B\ne\emptyset$, 
and then put a complete graph on these vertices.
$(T,\mathcal{B}')$ is the desired
tree decomposition of $Q^*$ on $\mathcal{H}$. 
Since each $H\in\mathcal{H}$
is connected and each bag in $\BB'$ induces a clique, it is clear that $(T,\mathcal{B}')$ is a valid tree decomposition. 
Further, since $(T,\mathcal{B})$ is $k$-sparse, it implies that
$(T,\mathcal{B}')$ has treewidth at most $k$.

\begin{lemma}
\label{lem:easydual}
For a connected graph system $(G,\mathcal{H})$ and a tree decomposition
$(T,\mathcal{B})$ of width $t$ that is $k$-sparse with respect to $\mathcal{H}$,
Algorithm {\bf $k$-SDS} computes a dual support
$Q^*$ with $\tw(Q^*)\le k$.
\end{lemma}
\begin{proof}
For a vertex $v\in V$, consider $\mathcal{H}_v$, the set of all subgraphs containing $v$. 
Since $(T,\mathcal{B})$ is a tree decomposition, at each node $x\in V(T)$
s.t. $v\in B_x$, $\mathcal{H}_v\subseteq B'_x$, where $B'_x$ is the bag at node $x$ in $\mathcal{B}'$,
i.e., corresponding to $B_x$. Since the vertices in $B'_x$ are pairwise adjacent, it implies
$\mathcal{H}_v$ induces a connected subgraph of $Q^*$. Hence,
$Q^*$ is a dual support.
\end{proof}

For the general setting, we obtain a dual support by \emph{sparsifying} the input graph system
so that it is $k$-sparse for some $k$, and such that a support on the sparsified graph system yields a 
support for the original graph system.
The sparsification yields a set
$\mathcal{H}'$ of subgraphs that satisfy the following properties:
$(i)$ they are in bijective correspondence with $\mathcal{H}$,
$(ii)$ each $H'\in\mathcal{H}'$ corresponding to $H\in\mathcal{H}$ is s.t.
$H'\subseteq H$. $(iii)$ If $H'\subset H$, then there is an $H''\in\mathcal{H}$
that \emph{pushed out} $H$.
In this case, $H'=H\setminus H''$ and there is an edge $e=\{u,v\}\in E(G)$
s.t. $u\in H'$ and $v\in H''$. We call $e$ a
\emph{connecting edge} between $H'$ and $H''$.

$\mathcal{H}'$ may contain multiple subgraphs spanning the same set of
vertices. We denote by $\unique(\mathcal{H}')$, a subset of $\mathcal{H}'$ consisting of one representative from each set of identical subgraphs.

\begin{lemma}
\label{lem:mvi}
Let $(G,\mathcal{H})$ be a non-piercing graph system, 
and $(T,\mathcal{B})$ a tree decomposition of $G$ of width $t$.
Then, there is a connected graph system $(G,\mathcal{H}')$ s.t.
$(T,\mathcal{B})$ is a $2^{4(t+1)}$-sparse with respect to $\unique(\mathcal{H}')$. If
$Q'$ is a dual support for $(G,\unique(\mathcal{H}'))$, there
is a dual support $Q^*$ for $(G,\mathcal{H}')$ s.t. $\tw(Q^*)=\tw(Q')$
and $Q^*$ is also a dual support for $(G,\mathcal{H})$.
\end{lemma}

\begin{proof}
Let $(T,\mathcal{B})$ be a tree decomposition of width $t$. We assume wlog
that $T$ is a binary tree rooted at $\rho$.
To obtain $\mathcal{H}'$, we
do a post-order edge traversal of $T$ and at each
edge, we do a \emph{pushing}. 

The pushing operation is as follows:
Let $z$ be the parent of a node $x$ in $T$.
Having done the pushing on the edges in $T_x$, we do the
following at $\{x,z\}$: 
For each $\emptyset\neq S\subseteq A_{xz}$, choose an $H\in\mathcal{M}_S$.
$H$ is called the \emph{pusher} for $S$ at $A_{xz}$. 
For each $H'\in\mathcal{H}^=_S$ s.t. $H'|_x\setminus H|_x\neq\emptyset$,
replace $H'$ by $H'|_x\setminus H|_x$ in $\mathcal{H}$. 
We say that $H'$ has been \emph{pushed}
by $H$ at $A_{xz}$. 

Observe that once a subgraph $H'$ is pushed by a pusher $H$
at $A_{xz}$, $H'|_x\setminus H|_x\subseteq G_x$, and since pushing is
done in a post-order edge traversal of $T$, 
$H'$ is pushed at most once.
Further, the subgraphs in
$\HH$ intersecting $A_{xz}$ were non-piercing before pushing and
since $H'|_x\setminus H|_x\neq\emptyset$, by part $(i)$ and $(ii)$ of Lemma~\ref{lem:npsep},
$H'|_{-x}\subseteq H|_{-x}$. Hence,
$H'|_x\setminus H|_x=H'\setminus H$. This implies
$H'|_x\setminus H|_x$ is a connected subgraph of $G$.
It follows that there
is a connecting edge $\{u,v\}\in E(G)$ s.t.
$u\in H'|_x\setminus H|_x$ and $v\in H|_x$ as
$H'|_x\setminus H|_x\neq\emptyset$ and $H'\cap H\neq\emptyset$.

Let $\mathcal{H}'$ be the subgraphs at the end of the algorithm.
Note that by construction $\mathcal{H}'$ is in bijective correspondence
with $\mathcal{H}$.
For $S\subseteq A_{xz}$,
if a subgraph $H'\in\HH^{=}_S$ was not pushed by a pusher $H$ at $A_{xz}$, then
$H'|_x=H|_x$. 
Therefore, the number of distinct subgraphs of $G_x$ in the collection $\mathcal{H}'$ intersecting $A_{xz}$ is less than
$2^{t+1}$. That is, 
\begin{align}
\label{eqn:uniq1}
|\unique(\mathcal{H}'_{A_{xz}}|_x)|<2^{t+1}
\end{align}

The subgraphs in $\unique(\HH')$ intersecting the bag $B_x$ can be 
associated with 4-tuples based on its intersection with $A_{xz}$, with the
adhesion sets between $x$ and its children, or with the bag $B_x$ itself.
From Eqn. (\ref{eqn:uniq1}), it follows that there are at most $2^{4(t+1)}$ distinct subgraphs in $\unique(\HH')$ intersecting $B_x$.

By the arguments above, $(G,\unique(\mathcal{H}'))$ is a connected
graph system that is at most $2^{4(t+1)}$-sparse, and hence by 
Lemma~\ref{lem:easydual}, it has a dual support $Q'$ of treewidth at most
$2^{4(t+1)}$. By Lemma~\ref{lem:contfree},
we can extend $Q'$ to a support $Q^*$ for $(G,\mathcal{H}')$ without
increasing the treewidth.

We now argue that $Q^*$ is a dual support for
$(G,\mathcal{H})$. Consider a vertex $v\in V(G)$. 
The algorithm above ensures that each subgraph is pushed at most once.
Suppose $H'\in\mathcal{H}_v$ was pushed by $H$ so that its modified copy $H''=H'\setminus H$
does not cover $v$. Then, $H\in\mathcal{H}_v$.
Let $e =\{a,b\}$ be the connecting edge between $H$ and $H''$ such that $a\in H$ and $b\in H''$.
Since $(T,\mathcal{B})$
is a valid tree decomposition, there is a bag $B$ containing both $a$ and $b$.
Let $H_1$ be the unique representative of $H''$ in $\unique(\HH')$.
Since Algorithm {\bf $k$-SDS} puts a complete graph on the subgraphs intersecting
$B$, it implies $H$ and $H_1$ are adjacent in $Q'$.
In $Q^*$, we made $H'$ adjacent to $H_1$.
Hence, $Q^*$ is a dual support for $(G,\HH)$.
\end{proof}

\dualsupporttw*
\begin{proof}
Let $(T,\mathcal{B})$ be a tree decomposition of $G$ of width $t$. 
If $(T,\mathcal{B})$ is $2^{4(t+1)}$-sparse, then we are done.
Otherwise, by Lemma~\ref{lem:mvi}, we obtain a dual support $Q^*$ for $(G,\mathcal{H})$,
of treewidth at most $2^{4(t+1)}$.
\end{proof}

\subsection{Intersection Support}\label{sec:intsupporttw}
In this section, we obtain an intersection support of treewidth $2^{O(2^{\tw(G)})}$ for a non-piercing intersection system $(G,\mathcal{H},\mathcal{K})$.
The construction of the intersection support uses the construction of both the primal and dual support and this leads to the double exponential bound
on the treewidth of the intersection support.

Let $(T,\mathcal{B})$ be a tree decomposition of $G$.
$(T,\mathcal{B})$ is said to be a $\KK$-\emph{easy} tree decomposition if
for each $K\in\mathcal{K}$ and each adhesion set $A$ with $A\cap K\neq\emptyset$, there is an $H\in\mathcal{H}_K$
s.t. $H\cap (K\cap A)\neq\emptyset$.
Similar to the setting for the dual support, we say that $(T,\mathcal{B})$ is $k$-sparse with respect to $\mathcal{H}$
if for each bag $B\in\mathcal{B}$, there are at most $k$ subgraphs in $\HH$ that intersect $B$.

We start by showing that if $(T,\mathcal{B})$ is a $\KK$-easy $k$-sparse tree decomposition, then there is an intersection support of treewidth at most $k$.
The proof follows along the same lines as the proofs of Lemma~\ref{lem:easytd} and Lemma~\ref{lem:easydual}. 

\begin{lemma}\label{lem:easyint}
Let $(G,\mathcal{H,K})$ be a connected intersection system.
If $(T,\mathcal{B})$ is a $\KK$-easy tree decomposition of $G$ s.t. it is $k$-sparse with respect to $\mathcal{H}$, 
then there is an intersection support for $(G,\mathcal{H,K})$ of treewidth at most $k$. 
\end{lemma}
\begin{proof}
We construct a tree decomposition $(T,\mathcal{B}')$ of a graph $\tilde{Q}$.
For each bag $B\in\mathcal{B}$ and each subgraph $K\in\KK$ intersecting $B$, we put a vertex in bag $B'\in\BB'$ for each $H\in\mathcal{H}_K$ s.t. $H$ intersects $B$, where $B'$ is the bag in $(T,\mathcal{B}')$ corresponding to the bag $B$ in $(T,\BB)$ . Since $(T,\mathcal{B})$ is $k$-sparse,
there are at most $k$ vertices of $\HH$ in $B'$. 
We put a complete graph on the vertices of $\mathcal{H}$ in $B'$ and obtain a tree decomposition $(T,\mathcal{B}')$ of a graph $\tilde{Q}$ on $\mathcal{H}$. Since each $H\in\mathcal{H}$ is connected, 
the vertex in $\tilde{Q}$ corresponding to $H$, lies in a connected set of bags of $(T,\mathcal{B}')$.
Further, by construction,
each edge between a pair of vertices $H, H'$ of $\tilde{Q}$ lies in some bag of $\mathcal{B}'$. Hence, $(T,\mathcal{B}')$ is a valid
tree decomposition of $\tilde{Q}$.

We now show that $\tilde{Q}$ is an intersection support.
Consider any $K\in\KK$.
Since $K$ is a connected subgraph of $G$, the vertices of $K$ lie in a sub-tree of $T$. 
Since $(T,\mathcal{B})$ is $\mathcal{K}$-easy, for each adhesion set intersected by $K$, there is a subgraph $H\in\mathcal{H}_K$
intersecting that adhesion set. This implies that $\mathcal{H}_K$ induces a connected subgraph of $\tilde{Q}$, and
hence $\tilde{Q}$ is an intersection support for $(G,\mathcal{H},\mathcal{K})$.
\end{proof}

We now show how we can modify any non-piercing system $(G,\mathcal{H,K})$ to one satisfying the conditions of Lemma~\ref{lem:easyint}.

\intsupporttw*
\begin{proof}
Let $(T,\mathcal{B})$ be a tree decomposition of $G$ of 
width $t$. Suppose $(T,\mathcal{B})$ is not $\mathcal{K}$-easy.
We define $\phi:V(G)\to\{\R,\B\}$.
For each $v\in V(G)$,  
if $\mathcal{K}_v\neq\emptyset$ and $\mathcal{H}_v\neq\emptyset$,
we set $\phi(v)=\B$. Otherwise, set $\phi(v)=\R$.
Under this coloring, since $(T,\mathcal{B})$ is not $\mathcal{K}$-easy,
there is a subgraph $K\in\mathcal{K}$ and an adhesion set $A$ s.t.
$K\cap A\subseteq\R(A)$. 
Since the subgraphs in $\mathcal{K}$ are non-piercing, 
by Lemma~\ref{lem:makeeasy}, we obtain a
tree decomposition $(T,\mathcal{B}')$ of treewidth $t'=2^{t+2}+t$
that is easy with respect to $\mathcal{K}$.
By the choice of coloring, it implies
$(T,\mathcal{B}')$ is $\mathcal{K}$-easy.

Each $K\in\mathcal{K}$ induces a connected subgraph of $G$.
Hence, $K$ intersects a connected set of bags of 
$(T,\mathcal{B'})$. Since $(T,\mathcal{B}')$ is $\mathcal{K}$-easy,
for any pair of vertices $u,v\in K$, there is a path in $T$
between a bag containing $u$ and a bag containing $v$ s.t.
each adhesion set on this path is intersected by a subgraph
$H\in\mathcal{H}_K$. 

Recall that to obtain dual support for $(G,\HH)$, the Algorithm {\bf $k$-SDS} adds a complete graph on the subgraphs
in $\mathcal{H}$ intersecting each bag of $(T,\mathcal{B}')$.
Therefore, a dual support for $(G,\HH)$ thus obtained, is also an intersection support for $(G,\HH,\KK)$.
However, this support may not have bounded treewidth. 
To obtain a support of small treewidth, we need to first sparisfy the subgraphs in $\mathcal{H}$.

Since the subgraphs in $\mathcal{H}$ are non-piercing, by Lemma~\ref{lem:mvi}, we can obtain a collection $\mathcal{H}'$
s.t. $(T,\mathcal{B}')$ is $2^{4(t'+1)}$-sparse with respect to 
$\unique(\mathcal{H}')$.
Note that $(T,\mathcal{B}')$ remains $\KK$-easy w.r.t. the intersection system $(G,\unique(\HH'),\KK)$.
Now, by Algorithm {\bf $k$-SDS} we
obtain a tree decomposition of a graph $Q^*$ that by Lemma~\ref{lem:easydual} is a dual support for $(G,\unique(\HH'))$ and $\tw(Q^*)\le 2^{4(t'+1)}=2^{{2^{t+4}}+4(t+1)}$.
Since $Q^*$ is obtained by {\bf $k$-SDS}, it also an intersection support 
for $(G,\unique(\mathcal{H}'),\mathcal{K})$.
By Theorem \ref{thm:dualsupporttw}, we can extend $Q^*$ to a dual support $\tilde{Q}$ for $(G,\HH)$ such that $\tw(\tilde{Q})=\tw(Q^*)$.
Then $\tilde{Q}$ is also an intersection support for $(G,\HH,\KK)$ since the tree decomposition $(T,\BB')$ is $\KK$-easy.
\end{proof}

\subsection{Implementation}
\label{sec:implementation}

In this section, we show that our algorithms for the construction of supports of bounded treewidth run in polynomial time if the treewidth of the host graph is bounded.
We argue by showing that if $\HH$ is a collection of non-piercing subgraphs of a bounded treewidth graph $G$, then the \emph{VC-dimension} of the corresponding set system defined by $(V(G),\HH)$ is bounded, which together with Sauer-Shelah lemma (Chapter 10 in \cite{Mat02LecDiscGeom}) implies that  $|\HH|$ is bounded by $\poly(n)$ where $n=|V(G)|$.

\begin{theorem}
\label{thm:vc-dimtw}
Let $(G,\HH,\KK)$ be a non-piercing intersection system of treewidth $t$.
Then $|\HH|,|\KK|=O(n^{3t+3})$.
\end{theorem}

\begin{proof}
We show that $|\HH|=O(n^{3t+3})$.
The bound on the size of $\KK$ follows analogously.

First, we will show that the VC-dimension of the set system defined by $(V(G),\HH)$ is at most $3t+3$.
Let $(T,\BB)$ be a tree decomposition of $G$ of width $t$.
We will show that no set of size more than $3t+3$ can be shattered.
Suppose not. Let $S$ be a set with $3t+4$ vertices that can be shattered.
There is an edge $\{x,y\}$ in $T$ such that the union of bags in $T_x$ contains a subset $S_1\subset S$ of at least $t+2$ vertices, and the union of bags in $T\setminus T_x$ contains a subset $S_2\subset S$ of at least $t+2$ vertices where $S_1\cap S_2=\emptyset$.
Note that such subsets should exist since each bag contains at most $t+1$ vertices and $|S|=3t+4$.

There are at least $k=\binom{t+2}{\lfloor\frac{t+2}{2}\rfloor}$ subsets of $S_i$ of size $\lfloor \frac{t+2}{2}\rfloor$ for $i\in\{1,2\}$.
Let $H_1,H_2,\ldots H_k$ be the distinct subgraphs in $\HH$ defined by these $k$ sets in $S_1$, and $H'_1,H'_2,\ldots H'_k$ be the distinct subgraphs in $\HH$ defined by these $k$ sets in $S_2$.
Note that $H_i$'s are pairwise incomparable in the containment order. Similarly, $H'_i$~'s are pairwise incomparable.

Next, we define a set of subgraphs in $\HH$ that intersect with both $S_1$ and $S_2$.
Such a set of subgraphs must exist since $S$ is shattered by element of $\HH$. 
Let $J_i\in\HH$ be a subgraph such that $J_i\cap S=H_i\cup H'_i$ for $i=1,2,\ldots k$.
Then, each $J_i$ contains exactly $2\lfloor\frac{t+2}{2}\rfloor$ vertices of $S$ and hence, they are pairwise incomparable in $S$.
Also, each $J_i$ intersects the adhesion set $A_e$ as they are connected subgraphs.
Since the subgraphs in $\HH$ are non-piercing, by part $(ii)$ of Lemma \ref{lem:npsep}, $J_i$ and $J_\ell$ should be incomparable in $A_e$ for $i\ne \ell$.
But there are only $\binom{t+1}{\lfloor\frac{t+1}{2}\rfloor}<k$ subsets of $A_e$ that are pairwise incomparable.
This contradicts the fact that the set $S$ can be shattered.
Hence, the VC-dimension of $(V(G),\HH)$ is at most $3t+3$.
By Sauer-Shelah lemma (see \cite{Mat02LecDiscGeom}, Chap. 10)
therefore, $|\HH|\le O(n^{3t+3})$. 
\end{proof}

Now, we are ready to show the running time of our algorithms.

\begin{theorem}\label{thm:runningtime-tw}
Let $G$ be an $n$-vertex graph of treewidth $t$ and $\HH,\KK$ be non-piercing subgraphs of $G$.
Then, a primal and a dual support of treewidth $O(2^t)$ and $O(2^{4t})$ respectively can be computed in time $\poly(n^t)$.
And, an intersection support of treewidth at most $2^{O(2^t)}$ can be computed in time $\poly(n^{2^t})$.
\end{theorem}

\begin{proof}
By Theorem \ref{thm:primalsupporttw}, \ref{thm:dualsupporttw} and \ref{thm:intsupporttw}, there are appropriate supports of the claimed treewidth. We show below the running time of the algorithms.

Since $\tw(G)=t$, a tree decomposition $(T,\BB)$ of width $t$ can be computed in time $2^{O(t)}n$ where $T$ is a binary tree and has $O(n)$ number of nodes~\cite{bodlaender1998partial,bodlaender2016,BODLAENDER1996358}.

For the construction of a primal support, we require an easy tree decomposition of $G$.
We did a post-order traversal of the edges in $T$ 
and chose a minimal subgraph in $\mathcal{M}_S$ for each non-empty subset $S$ of an adhesion set $A$ in $T$.
Since $A$ has at most $t+1$ vertices, there are at most $2^{t+1}$ choices for $S$.
Also, by Theorem \ref{thm:vc-dimtw}, $|\HH|=n^{3t+3}$.
Therefore, a primal support of treewidth $2^{t+2}+t$ can be computed in time $\poly(n^t)$.

For a dual support, we first construct a graph system that is $2^{4(t+1)}$-sparse.
For such a construction, we again do a post-order traversal of the edges in $T$ and choose a subset $S$ of an adhesion set $A$ to select a pusher in $\mathcal{M}_S$.
In $O(|\HH|^2)$ time, we can do the pushing operation at a subset $S$ to get new subgraphs.
Therefore all the adhesion sets can be processed in time $O(2^tn|\HH|^2)$.
Once a $2^{4(t+1)}$-sparse tree decomposition is obtained, the algorithm {\bf $k$-SDS} puts a complete graph on the subgraphs in each bag and hence can be done in $O(k^2t)$ time where $k=2^{4(t+1)}$.
Therefore, a dual support of treewidth $2^{4(t+1)}$ can also be computed in time $\poly(n^t)$ since $|\HH|\le n^{3t+3}$ by Theorem \ref{thm:vc-dimtw}.

Finally, the construction of an intersection support follows the construction of primal and dual supports.
By the arguments above, it follows that a $\KK$-easy tree decomposition can be computed in time $\poly(n^t)$ where the width of the resulting tree decomposition is $O(2^t)$.
Then a $2^{O(2^t)}$-sparse tree decomposition can be computed in time $\poly(n^{2^t})$.
Hence, the total running time to compute an intersection support is $\poly(n^{2^t})$.
\end{proof}

\section{Lower Bounds}\label{sec:lowerbounds}
In this section, we show that there exist non-piercing graph systems for which the treewidth of any primal or dual support is exponential in the treewidth of the host graph.

\primallb*
\begin{proof}
Let $n = \lfloor m/2\rfloor$ and let $N = \binom{n}{\lfloor n/2 \rfloor}$.
We construct an $N\times N$ grid $B$ of isolated vertices $b_{i,j}$, $i=1,\ldots, N$ and $j=1,\ldots, N$.
This constitutes the set of vertices colored $\B$.
Let $R=\{r_1,\ldots, r_n\}$ and $C=\{c_1,\ldots, c_n\}$ be two sets of $n$ isolated vertices each.
The vertices in $R\cup C$ are colored $\R$.
Let $\mathcal{R}=\{R_1,\ldots, R_{N}\}$ be the collection of $\binom{n}{\lfloor n/2 \rfloor}$ subsets of $R$ of size $\lfloor n/2 \rfloor$.
Similarly, let $\mathcal{C}=\{C_1,\ldots, C_{N}\}$ be the $\binom{n}{\lfloor n/2 \rfloor}$ subsets of $C$ of size $\lfloor n/2 \rfloor$.
Let the host graph $G$ be the complete bipartite graph with bipartition $(R\cup C)$ and $B$.

For $i=1,\ldots, N$ and $j=1,\ldots, N-1$, let $H_{i;j,j+1}$ be the subgraph induced on the vertices $\{b_{i,j}, b_{i,j+1}\}\cup R_i\cup C_{j+1}$ for $i=1,\ldots, N$ and $j=1,\ldots, N-1$.
Similarly, for each $i=1,\ldots, N-1$ and $j=1,\ldots, N$, let $H_{i,i+1;j}$ be the subgraph induced on the vertices $\{b_{i,j}, b_{i+1,j}\}\cup R_{i+1}\cup C_j$.
See Figure \ref{fig:primaltw1}.
Let $\mathcal{H}=\{H_{i;j,j+1}:i=1,\ldots, N, j=1,\ldots, N-1\}\cup\{H_{i,i+1;j}:i=1,\ldots, N-1, j=1,\ldots, N\}$.

By construction, each subgraph in $\mathcal{H}$ contains exactly two vertices colored $\B$ corresponding to vertices in the grid that are
consecutive row-wise or column-wise, and one subset from each of the collections $\{R_1,\ldots, R_N\}$ and $\{C_1,\ldots, C_N\}$.
Since $G$ is a complete bipartite graph, for any $H\in\mathcal{H}$, 
the two blue vertices in $H$ are adjacent to all red vertices in $H$. Hence, $H$ is a connected subgraph of $G$.

Consider two distinct subgraphs $H,H'\in\mathcal{H}$. Since $H$ and $H'$ differ in either their row index or column index, 
$H\setminus H'$ and $H'\setminus H$ each contains at least one blue vertex.
Suppose $H$ and $H'$ differ in their row index.
By construction, $H$ and $H'$ are adjacent to distinct subsets of size $\lfloor n/2\rfloor$ in $R$. Hence, 
both $(H\setminus H')\cap R$ and $(H'\setminus H)\cap R$ are non-empty. 
Similarly, if $H$ and $H'$ differ in their column index, then $(H\setminus H')\cap C$ and $(H'\setminus H)\cap C$
are non-empty. Since $H\setminus H'$ contains at least one blue vertex of $H$ and 
$G$ is a complete bipartite graph,
each red vertex in $H\setminus H'$ is connected to the blue vertices in $H\setminus H'$. On the other hand, if both
blue vertices of $H$ are present in $H\setminus H'$, then they are adjacent via a red vertex in $H\setminus H'$ 
since $H\setminus H'\cap R$ or $H\setminus H'\cap C$ is non-empty and $G$ is a complete bipartite graph. 
Thus, $H\setminus H'$ is connected. A symmetric argument implies $H'\setminus H$ is
connected and hence $\mathcal{H}$ is a set of non-piercing subgraphs of $G$.
Since $G$ is a complete bipartite graph, each blue vertex in $H\setminus H'$ is adjacent to each red vertex in $H\setminus H'$.
Thus, $H\setminus H'$ is connected.
Similarly, $H'\setminus H$ is connected and hence, $(G,\HH)$ is a non-piercing graph system.

Since $G$ is a bipartite graph with $2n$ red vertices, $G$ has treewidth at most $2n\le m$. There is a subgraph in $\mathcal{H}$
corresponding to each pair of blue vertices that are consecutive along a row of the grid or along a column of the grid.
Therefore, any primal support contains a grid of size $N\times N$ as a subgraph.
Since $N=\binom{n}{\lfloor n/2 \rfloor}\ge \frac{2^n}{\sqrt{2n}}$, it follows that the treewidth of any primal support is at least $\frac{2^n}{\sqrt{2n}} \ge \frac{2^{\lfloor m/2\rfloor}}{\sqrt{m}}$.  
\end{proof}

\begin{figure}[ht!]
\centering
\begin{subfigure}{0.47\textwidth}
\begin{center}
\includegraphics[scale=.65]{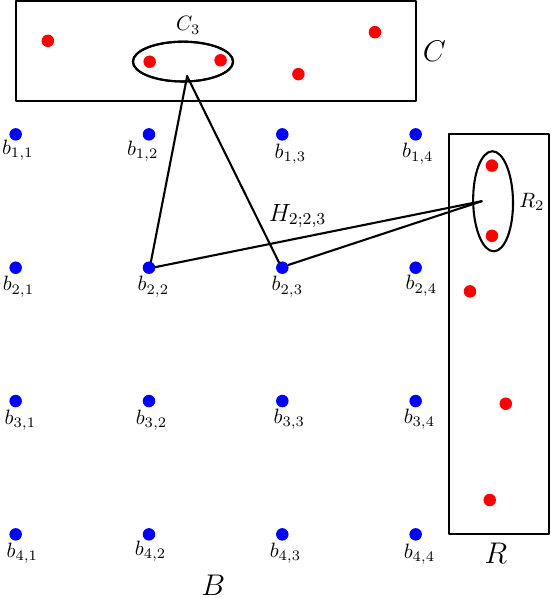}
\caption{Construction of primal hypergraph}
\label{fig:primaltw1}
\end{center}
\end{subfigure}
\hfill
\begin{subfigure}{0.45\textwidth}
\begin{center}
\includegraphics[scale=.65]{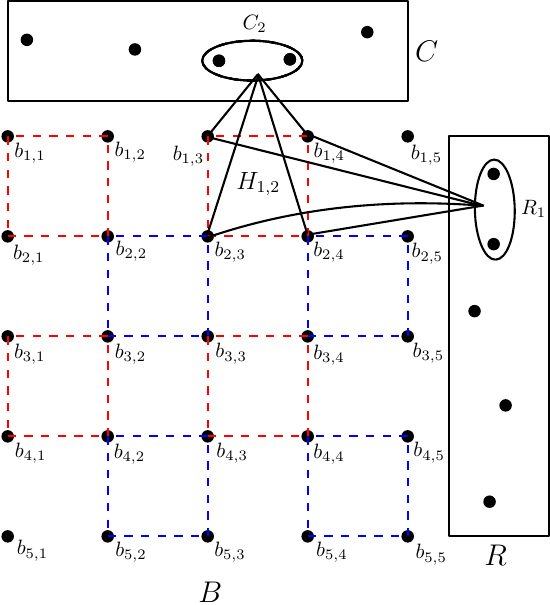}
\caption{Construction of dual hypergraph}
\label{fig:dualtw1}
\end{center}
 \end{subfigure}
 \caption[Lower bounds on treewidth of primal and dual supports]{Construction of exp. lower bounds on treewidth for primal and dual supports. (a) Subgraph $H_{2;2,3}$ consists of vertices $\{b_{2,2},b_{2,3}\}\cup R_2\cup C_3$. (b) Subgraph $H_{1,2}$ consists of vertices $\{b_{1,3},b_{2,3},b_{1,4},b_{2,4}\}\cup R_1\cup C_2$.}
 \label{fig:lowerbounds}
 \end{figure}

\duallb*
\begin{proof}
The construction is similar to that of the primal lower bound. 
Let $n=\lfloor m/2\rfloor$ and let $N = \binom{n}{\lfloor n/2 \rfloor}$.
Let $B$ be a $(2N+1)\times (2N+1)$ grid of isolated vertices $b_{i,j}$ for $i,j\in\{1,\ldots 2N+1\}$. Let $R$ and $C$ be two
sets of $n$ isolated vertices in each.
The host graph is the complete bipartite graph $G$ with bipartition 
$(R\cup C)$ and $B$.
Let $\mathcal{R}=\{R_1,\ldots, R_N\}$ denote the $\binom{n}{\lfloor n/2 \rfloor}$ subsets of $R$ of size $\lfloor n/2 \rfloor$.
$\mathcal{C}=\{C_1,\ldots, C_N\}$ is defined analogously.

We define the subgraphs $H_{i,j}$ so that each subgraph contains 4 vertices of $B$ in a checkerboard pattern.
More formally, the subgraphs $H_{i,j}$ are defined as follows.
Let $i\in\{1,\ldots, 2N\}$.
If $i$ is odd, then for $j=1,\ldots, N$, $H_{i,j}$ is the induced graph on 
$b_{i,2j-1}, b_{i+1,2j-1}, b_{i,2j}, b_{i+1,2j}$ and the subset
$R_i$ of $R$  and $C_j$ of $C$.
If $i$ is even, then for $j=2,\ldots, N$,
$H_{i,j}$ is the induced subgraph of $G$ on the vertices
$b_{i,2j}, b_{i+1,2j}, b_{i,2j+1}, b_{i+1,2j+1}$ and the subsets $R_i$ and $C_j$.
See Figure \ref{fig:dualtw1}. For odd (even) $i$, the vertices of $H_{i,j}$ in $B$ are shown by red (blue) squares.
Let $\mathcal{H}$ denote the set of subgraphs thus constructed. 

Each subgraph $H_{i,j}\in\HH$ is induced on four vertices of $B$ and a subset $R_i$ of $R$ and a subset $C_j$ of $C$.
Since $G$ is a complete bipartite graph, it follows that each subgraph $H\in\mathcal{H}$ induces a connected
subgraph of $G$. Consider any two subgraphs $H, H'\in\mathcal{H}$. Since they differ in at least one of the row
or column indices, it follows that $H\setminus H'$ and $H'\setminus H$ each contain at least three 
vertices of $B$. Further, if $H$ and $H'$ differ in their row index, then they contain distinct
subsets in $\mathcal{R}$. Similarly, if $H$ and $H'$ differ in their column index, they contain distinct subsets in $\mathcal{C}$. Since no two subsets in $\mathcal{R}$ are contained in one another,
it follows that $(H\setminus H')\cap R$ and $(H'\setminus H)\cap R$ are non-empty if $H$ and $H'$ differ
in their row index. Similarly, $(H\setminus H')\cap C$ and $(H'\setminus H)\cap C$ are non-empty if
$H$ and $H'$ differ in their column index. Since $G$ is a complete bipartite graph and $H\setminus H'$
contain vertices of $B$ and vertices of either $R\cup C$, it follows that $H\setminus H'$ is
connected. Similarly, $H'\setminus H$ is connected. Hence, $\mathcal{H}$ is a collection of non-piercing
subgraphs of $G$.

Since $|R\cup C|=2n$, it follows that $\tw(G)\le 2n\le m$. On the
other hand, for $i=2,\ldots, 2N$ and $j=2,\ldots, 2N$, the vertices $b_{i,j}$ are contained in exactly two subgraphs in $\mathcal{H}$, and thus, they have to be adjacent in any dual support for $(G,\HH)$. This implies that any dual support $Q^*$ must
contain a grid of size $N\times N$ as an induced subgraph and therefore, $\tw(Q^*)\ge N$.
Since $N=\binom{n}{\lfloor n/2\rfloor}$, $\tw(Q^*)\ge \frac{2^{n}}{\sqrt{2n}} \ge \frac{2^{\lfloor m/2\rfloor}}{\sqrt{m}}$.
\end{proof}

\section{Outerplanar Support}\label{sec:outerplanar}
In this section, we give a polynomial time algorithm
that computes an intersection support for an outerplanar non-piercing intersection system.
Outerplanar graphs have treewidth at most 2.
Let $G$ be an outerplanar graph and $\HH,\KK$ be collections of non-piercing subgraphs of $G$.
By Theorem \ref{thm:primalsupporttw} and Theorem \ref{thm:dualsupporttw}, the graph system $(G,\HH)$ admits primal and dual supports of treewidth at most $2^{4}+2$ and $2^{12}$ respectively.
Further, by Theorem \ref{thm:intsupporttw}, $(G,\HH,\KK)$ admits an intersection support of treewidth at most $2^{2^7}$.
However, we show that in this restricted setting, i.e., when $G$ is an outerplanar graph, there is an intersection support that is outerplanar.
This also implies outerplanar primal and dual supports.

Recall from Section~\ref{sec:prelim} that we can assume $G$ is maximal outerplanar. 
Let $C$ be the outer face in an outerplanar embedding of $G$.
Let $\HH,\KK$ be two families of connected subgraphs of $G$.
Observe that two subgraphs of a graph intersect if and only if they share a vertex. Hence, the hypergraph on $\mathcal{H},\mathcal{K}$ induced on $C$ is identical to that induced on $G$.
Consequently, it is sufficient to construct an outerplanar support for the cycle systems that are \emph{strong $axax$-free}, a
notion we define below.

\subsection{Cycle $axax$-free systems}\label{sec:cycleaxax}

For a cycle $C$ embedded in the plane, and collections $\HH$ and $\KK$ of subgraphs (not necessarily connected) of $C$, we call $(C,\HH)$ a \emph{cycle system} and $(C,\HH,\KK)$ a \emph{cycle intersection system}.
We start with the definitions of cycle systems that 
(\emph{strong}) $axax$-free.

\begin{definition}[$axax$-free]
\label{defn:axax}
Let $(C,\mathcal{H})$ be a cycle system.
$H,H'\in\mathcal{H}$ are an $axax$-pair if there are four distinct vertices 
$a_1,x_1,a_2,x_2$ in cyclic order on $C$ such that $a_1,a_2\in H\setminus H'$ and $x_1,x_2\in H'$.
$(C,\mathcal{H})$ is $axax$-free if there are no $axax$-pairs in $\mathcal{H}$.
For two families $\HH,\KK$ of subgraphs of $C$, the intersection system $(C,\mathcal{H},\mathcal{K})$ is \emph{$axax$-free} if both $(C,\mathcal{H})$ and $(C,\mathcal{K})$ are $axax$-free.
\end{definition}

\begin{definition}[Intersection property] The  cycle intersection system $(C,\mathcal{H},\mathcal{K})$ is said to satisfy the \emph{intersection property} if
for any $H\in\mathcal{H}$ and $K\in\mathcal{K}$ with 
four vertices $h_1, k_1, h_2, k_2$ in cyclic order on $C$ such that $h_1, h_2\in H$ and $k_1, k_2\in K$, then $H\cap K\neq\emptyset$.
\end{definition}

\begin{definition}[(Strong) $axax$-free] 
The cycle intersection system $(C,\HH,\KK)$ is said to satisfy the \emph{strong $axax$-free property}, if it is $axax$-free and it satisfies the intersection property.
\end{definition}

\begin{restatable}{lemma}{npaxaxfree}
\label{lem:npaxax-free}
Let $(G,\mathcal{H},\mathcal{K})$ be an embedded outerplanar non-piercing system with $C$ denoting the outer cycle of $G$.
Then, $(C,\mathcal{H},\mathcal{K})$ satisfies the strong $axax$-free property.    
\end{restatable}

\begin{proof}
Let $H,H'\in\mathcal{H}$.
Since $\mathcal{H}$ are non-piercing subgraphs of $G$, $H\setminus H'$ and $H'\setminus H$ are both connected.  
Suppose there is a cyclic sequence $a_1, x_1, a_2, x_2$ of vertices on $C$ such that $a_1, a_2\in H\setminus H'$, and $x_1, x_2\in H'$.
Since $H\setminus H'$ is connected, there is a path $P$ between $a_1$ and $a_2$ lying in $H\setminus H'$. Then, $x_1$ and $x_2$ lie in separate components of $G\setminus P$, contradicting the assumption that $H'\setminus H$ is connected.
An identical argument shows that $(C,\mathcal{K})$ is $axax$-free.
See Figure~\ref{fig:axax-free}. 

To see that $(C,\mathcal{H},\mathcal{K})$ satisfies the intersection property, consider $H\in\mathcal{H}$ and
$K\in\mathcal{K}$ with vertices $h_1, k_1, h_2, k_2$ in cyclic order s.t. $h_1, h_2\in H$ and $k_1, k_2\in K$.
Since $H$ is connected in $G$, there is a path $P$ 
between $h_1$ and $h_2$, all of whose vertices lie in $H$. 
If $H\cap K=\emptyset$, then $P$ lies in $H\setminus K$. Since
$h_1$ and $h_2$ are non-consecutive on $C$, $G\setminus P$ gets
separated into two components with $k_1$ and $k_2$ in distinct
components. This implies there is no path between $k_1$ and
$k_2$ in $G$ that lies entirely in $K$, a contradiction since $K$ is connected.
\end{proof}

\begin{figure}[ht!]
    \centering
    \includegraphics[width=0.3\linewidth]{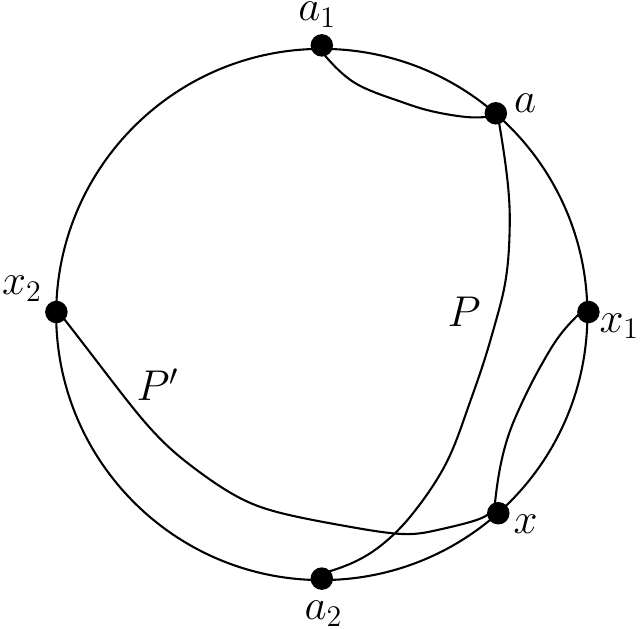}
    \caption{Non-piercing implies $axax$-free.}
    \label{fig:axax-free}
\end{figure}

Our next theorem shows that if $(C,\HH,\KK)$ is strong $axax$-free, then it admits an outerplanar intersection support.
We will use this to prove the existence of an outerplanar support for the non-piercing intersection system $(G,\HH,\KK)$.

\begin{restatable}{theorem}{outintsupport}
\label{thm:outintsupport}
Let $(C,\mathcal{H},\mathcal{K})$ be a strong $axax$-free outerplanar system. Then, there is an outerplanar support $Q$ for $(C,\HH,\KK)$.
\end{restatable}

Now, Theorem~\ref{thm:np-outintsupport} is a direct implication of Lemma~\ref{lem:npaxax-free} and Theorem~\ref{thm:outintsupport}.

\outerplanar*
\begin{proof}
Let $C$ be the outer cycle in an outerplanar embedding of $G$.
By Lemma~\ref{lem:npaxax-free}, $(C,\HH,\KK)$ is strong $axax$-free.
By Theorem~\ref{thm:outintsupport}, there is an outerplanar support $\tilde{Q}$ for $(C,\HH,\KK)$.
$\tilde{Q}$ is also a support for $(G,\HH,\KK)$ since the underlying intersection hypergraph defined by $\HH$ and $\KK$ remains the same on $G$ and on $C$.
\end{proof}

For Theorem~\ref{thm:np-outintsupport} to go through,
it remains to prove Theorem~\ref{thm:outintsupport} which we do in Section \ref{sec:construction-opsupport}. Below, we show that we can simplify an $axax$-free cycle system so that each vertex of $C$ is contained in at least one subgraph in $\mathcal{H}$ and
one subgraph in $\mathcal{K}$.

\begin{definition}[Reduced cycle system]\label{defn:reducedsystem}
A cycle system $(C,\HH,\KK)$ is called \emph{reduced} if for each vertex $v$ of $C$, there is an $H\in\HH$ and a $K\in\KK$ such that $v\in V(H)\cap V(K)$.
\end{definition}

Consider a cycle system $(C,\mathcal{H,K})$.
Let $C'$ be the induced cycle obtained from $C$ by removing each vertex $v$ of $C$ and making its neighbours adjacent, whenever $v\notin V(H)$ for all $H\in\HH$ or $v\notin V(K)$ for all $K\in\KK$.
Let $\mathcal{H}'=\{H\cap C':H\in\mathcal{H}\}$ and $\mathcal{K}'=\{K\cap C':K\in\mathcal{K}\}$.
It follows that $(C',\HH',\KK')$ is a reduced cycle system.
The following proposition shows that it is sufficient to construct a support when $(C,\mathcal{H,K})$ is a reduced system.

\begin{proposition}\label{prop:reducedsystem}
If $(C,\mathcal{H,K})$ is a strong $axax$-free system, 
then the reduced system $(C',\mathcal{H',K'})$ is also a strong $axax$-free system.
Further, any intersection support for $(C',\mathcal{H',K'})$ is also an intersection support for $(C,\mathcal{H,K})$.
\end{proposition}
\begin{proof}
If $(C,\HH,\KK)$ is strong $axax$-free, removing an $\mathcal{H}$-vertex or a $\mathcal{K}$-vertex leaves the reduced system $(C',\HH',\KK')$ is strong 
$axax$-free and the resulting intersection hypergraph remains the same since we did not remove any vertex $v\in H\cap K$ for any $H\in\HH$ and $K\in\KK$.
Hence, a support for the reduced system is also a support for the original system.
\end{proof}

As a consequence of Proposition~\ref{prop:reducedsystem}, 
we assume throughout that $(C,\mathcal{H},\mathcal{K})$ is a reduced system.

\subsection{Construction of Outerplanar Supports}\label{sec:construction-opsupport}
In this section, we prove Theorem \ref{thm:outintsupport}.
We start with some basic terminology required for the proof.
Let $C=\{0,\ldots, n-1\}$ be a cycle on $n$ vertices oriented clockwise, and let $\mathcal{R}$ be a collection of arcs on $C$ whose both ends are defined by the vertices of $C$ such
that no two arcs in $\mathcal{R}$ contain the same set of vertices.

For an $R\in\mathcal{R}$, if $R=\arc[i,j]$, i.e., $R$ consists of a 
consecutive sequence of vertices, i.e., 
$R=[i, i+1, \ldots, j]$ where the indices are numbered~$\pmod n$, 
we say that $R$ is a \emph{run} on $C$. We also use $\arc(i,j)$ to 
denote the run $[i+1,\ldots, j-1]$. 
Let $s(R)=i$ and $t(R)=j$.
Consider a pair of arcs $R, R'\in\mathcal{R}$ such that $R\subseteq R'$. In a traversal of $C$ starting at $s(R')$, if we have 
$s(R')<s(R)<t(R)<t(R')$, 
then we say that $R$ is \emph{strictly contained} in $R'$. If $s(R)=s(R')$, or 
$t(R)=t(R')$, then we say that $R$ is \emph{weakly contained} in $R'$. 
If there is no pair of arcs $R, R'\in\mathcal{R}$ such that $R$ is strictly contained in $R'$, then we say that
$\mathcal{R}$ is \emph{strict-containment free}. 
If there is no pair of arcs $R, R'\in\mathcal{R}$
that are weakly contained in each other or strictly contained in each other, then $\mathcal{R}$ is \emph{containment-free} i.e., $R\setminus R'\ne\emptyset$ for all $R,R'\in\mathcal{R}$. See Figure
\ref{fig:singleruns}.

\begin{figure}[ht!]
\begin{subfigure}{0.45\textwidth}
\begin{center}
\includegraphics[scale=.9]{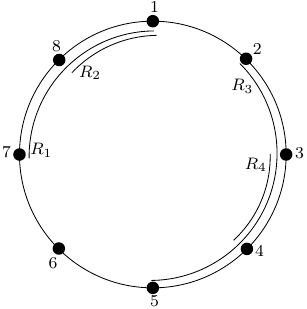}
 % \vspace{0.2cm}
\caption{$R_2$ is weakly contained in $R_1$, and $R_4$ is strictly contained in $R_3$.}
\label{fig:weakandstrict}
\end{center}
\end{subfigure}
\hfill
\begin{subfigure}{0.45\textwidth}
\begin{center}
\includegraphics[scale=.9]{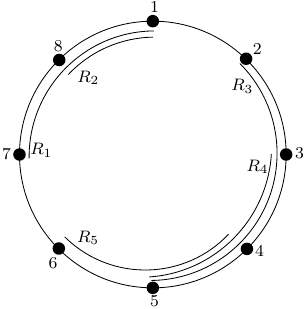}
% \vspace{0.85cm}
\caption{\vspace{0.5cm}Strict-containment free}
\label{fig:strictcontfree}
\end{center}
 \end{subfigure}
 % \vspace{0.5cm}
 \caption{Weak containment, strict containment and strict-containment free. Induced subgraphs of the cycle are shown by corresponding arcs inside the cycle.}
 \label{fig:singleruns}
 \end{figure}

Strict containment defines a natural partial order on $\mathcal{R}$.
Let $\mathcal{R}^*\subseteq\mathcal{R}$ denote the maximal elements of this \emph{strict-containment order}.
Then, $\mathcal{R^*}$ is a maximal strict-containment free subset of $\mathcal{R}$.
Note that $\mathcal{R}^*$ can contain subgraphs that are weakly contained in each other.

An immediate consequence of a cycle system $(C,\mathcal{H},\mathcal{K})$ being $axax$-free that will be useful later
is the following proposition which is a 
slight variation of Lemma~\ref{lem:contfree}. 
The proof is along the same lines as of Lemma~\ref{lem:contfree}.

\begin{proposition}
\label{prop:wcontfree}
Let $(C,\mathcal{H,K})$ be a cycle $axax$-free system where subgraphs in $\mathcal{H}$ are single runs, and let $\mathcal{H}^*$ be the maximal elements
in the strict-containment order on $\mathcal{H}$.
If $(C,\mathcal{H^*,K})$ admits an outerplanar support, then so does $(C,\mathcal{H,K})$.
\end{proposition}

As a consequence of Lemma \ref{lem:contfree}, or Proposition \ref{prop:wcontfree}, we can assume that $\mathcal{H}$ is either
containment-free, or it is strict-containment free (as will be useful in our proofs).

Let $\mathcal{R}$ be a collection of arcs on a cycle $C$ that are strict-containment free. 
Starting at an arbitrary vertex on $C$, in a clockwise traversal, we construct a cycle $C(\mathcal{R})$ on
$\mathcal{R}$ as follows: 
When we visit a vertex $v$ of $C$, for each $R\in\mathcal{R}$ 
s.t. $s(R)=v$,
we put a vertex in $C(\mathcal{R})$. The vertices added to $C(\mathcal{R})$ w.r.t. $v$
are ordered in increasing order of $t(R)$, with ties broken
arbitrarily.
We call the cycle $C(\mathcal{R})$ thus constructed, the
cycle in \emph{lex. cyclic order}, and we say that $\mathcal{R}$
is ordered in \emph{lex. cyclic order}.

In a cycle system $(C,\mathcal{H},\mathcal{K})$, each subgraph $X\in\mathcal{H}\cup\mathcal{K}$ induces
a sequence of runs on $C$. Let $n_X$ denote the number of runs of $X$ in $C$.
Let $r_0(X)=[s_0,\ldots, t_0], r_1(X)=[s_1,\ldots, t_1], \ldots, r_{k-1}=[s_{k-1},\ldots, t_{k-1}]$ denote the
$k=n_X$ runs of $X$ on $C$. 
For $i=0,\ldots, k-1$, let $d_i$ denote the \emph{chord} $\{t_i, s_{i+1}\}$. 
Let $\ell(d_i)=|\arc[t_i, s_{i+1}]|$, and let 
$\ell(X) = \arg\min_{d_i} \ell(d_i)$, with ties broken arbitrarily. If $n_X=1$, then $d_X$ is undefined.

The proof of Theorem \ref{thm:outintsupport} proceeds in \emph{three easy steps}.
We start with the case when all subgraphs in $\mathcal{H}\cup \mathcal{K}$ induce single runs on $C$. 
We next use this result to 
obtain an outerplanar support for the case when only the
subgraphs in $\mathcal{H}$ are required to induce single runs on $C$.
Finally, we use the outerplanar supports guaranteed by the two special
cases to obtain an outerplanar support for the general case.

Let each subgraph in $\mathcal{H}\cup\mathcal{K}$ induce a single run
on $C$. For the requirement in the general settings, we assume by Proposition~\ref{prop:wcontfree} that $\mathcal{H}$ is strict-containment free. We claim that $C(\mathcal{H})$, the cycle on
$\mathcal{H}$ in lex. cyclic order is the desired support for $(C,\mathcal{H},\mathcal{K})$.

\begin{lemma}\label{lem:cyclesupport}
Let $(C,\mathcal{H,K})$ be a cycle system such that $\mathcal{H}$ is strict-containment free and each $X\in\mathcal{H\cup K}$ induce a single run on $C$.
Then, the cycle $C(\mathcal{H})$ is an outerplanar support for $(C,\mathcal{H,K})$.
\end{lemma}

\begin{proof}
$C(\mathcal{H})$ is clearly an outerplanar graph. Consider a subgraph $K\in\mathcal{K}$. Since $K$ induces a single run on $C$ 
and the subgraphs in $\mathcal{H}$ are strict-containment free, 
it follows that the subgraphs in $\mathcal{H}_K$ appear consecutively in lex. cyclic order on $C(\mathcal{H})$ and thus induce a connected subgraph of $C(\HH)$.
\end{proof}

Now, assume that subgraphs in $\mathcal{H}$ induce single runs on $C$, and those in $\KK$ can have multiple runs.
We claim that there is an outerplanar support with outer cycle $C(\mathcal{H})$. But, before we do that, we start with the following
simple consequence of $(C,\mathcal{H},\mathcal{K})$ being $axax$-free.

\begin{proposition}\label{prop:chordpartition}
Let $(C,\mathcal{H})$ be $axax$-free. For an $H\in\mathcal{H}$, let $u,v$ be two non-consecutive vertices on $C$ with $u,v\in H$.
Then, for any $H'\in\mathcal{H}$ s.t. $H'\cap \arc(u,v)\neq\emptyset\neq H'\cap \arc(v,u)$, $H'$ must contain either $u$ or $v$.
Moreover, if $H\cap \arc(u,v)=\emptyset$, then $H'\cap \arc[v,u]\subseteq H\cap \arc[v,u]$.
\end{proposition}
\begin{proof}
Let $x $ and $y$ be any two vertices of $H'$ in $\arc(u,v)$ and $\arc(v,u)$ respectively, and let $u,v\in H\setminus H'$.
The vertices $u,x,v,y$ form an $axax$-pattern on $C$; a contradiction since $(C,\mathcal{H})$ is $axax$-free.

The proof for the \emph{moreover} part is similar.
Indeed, let $h'\in \arc[v,u]$ s.t. $h'\in H'\setminus H$.
Also, note that $x\in H'\setminus H$. Hence, the subgraphs $H,H'$ form an $axax$-pattern as witnessed by the cyclic sequence $h',u,x,v$; a contradiction.
\end{proof}

\begin{lemma}\label{lem:basecase}
Let $(C,\mathcal{H,K})$ be any $axax$-free cycle system s.t. each $H\in\mathcal{H}$ induces 
a single run on $C$.
If $\mathcal{H}$ is strict-containment free, then there exists an outerplanar support for $(C,\mathcal{H,K})$ with outer cycle $C(\mathcal{H})$.
\end{lemma}
\begin{proof} 
Let $N=N(C,\mathcal{K}) = \sum_{K\in\mathcal{K}}(n_K-1)$, where $n_K$ is the number of runs of $K$ on $C$.
We proceed by induction on $N$. If $N=0$, then $n_K=1$ for all $K\in\mathcal{K}$ and we are done by Lemma~\ref{lem:cyclesupport}.
So, suppose $N\ge 1$ i.e., $n_K\ge 2$ for some $K\in\mathcal{K}$.

We assume the lemma holds when $N(C,\mathcal{K}')<N$ for any $axax$-free system $(C,\mathcal{H,K'})$.
Let
$(C,\mathcal{H,K})$ be
an $axax$-free system with $N(C,\mathcal{K})=N$.
Let $K_0=\arg\min_{K\in\mathcal{K}}\ell(K)$.
Let $d_{K_0}=\{u_0, v_0\}$ be the chord of $K_0$ realizing this minimum.
We use the chord $d_{K_0}$ to split $C$
into two cycles.

Let $C_L$ denote the cycle $\arc[v_0, u_0]\cup d_{K_0}$, and $C_R$ denote the cycle $\arc[u_0, v_0]\cup d_{K_0}$. 
Let
$\mathcal{K}_X=\{K\cap C_X: K\in\mathcal{K}\}$, where
$X\in\{L, R\}$. Observe that $K_0$ appears both in $C_L$
and $C_R$, where in $C_R$, $K_0$ spans the vertices 
$\{u_0, v_0\}$.

Let $H_1,\ldots, H_r$ be the subgraphs in $\mathcal{H}_{u_0}$ 
in lex. cyclic order, i.e., 
labeled such that
$(H_r\cap C_L)\subseteq (H_{r-1}\cap C_L)\subseteq\ldots\subseteq (H_1\cap C_L)$.
Since $\mathcal{H}$ is strict-containment free, it follows that $(H_1\cap C_R)\subseteq(H_2\cap C_R)\subseteq\ldots\subseteq(H_r\cap C_R)$. 
Similarly, if $H'_1,\ldots, H'_s$ are the subgraphs in $\mathcal{H}_{v_0}$ in lex. cyclic order,
then $(H'_1\cap C_L)\subseteq (H'_2\cap C_L)\subseteq\ldots\subseteq (H'_s\cap C_L)$, and 
$(H'_s\cap C_R)\subseteq (H'_{s-1}\subseteq C_R)\ldots\subseteq (H'_1\cap C_R)$. Figure~\ref{fig:leftrightordering} shows the
subgraphs in $\mathcal{H}_{u_0}$ and in $\mathcal{H}_{v_0}$.
 
We let $\mathcal{H}_L=\{H\in\mathcal{H}: H\subseteq \arc(v_0, u_0)\}\cup\{H_1\cap C_L\}\cup\{H'_s\cap C_L\}$, 
and $\mathcal{H}_R = \{H\cap C_R: H\in\mathcal{H}\}$.  See Figure~\ref{fig:constructing H_L and H_R}.

\begin{figure}[h!]
    \centering
    \includegraphics[width=0.8\linewidth]{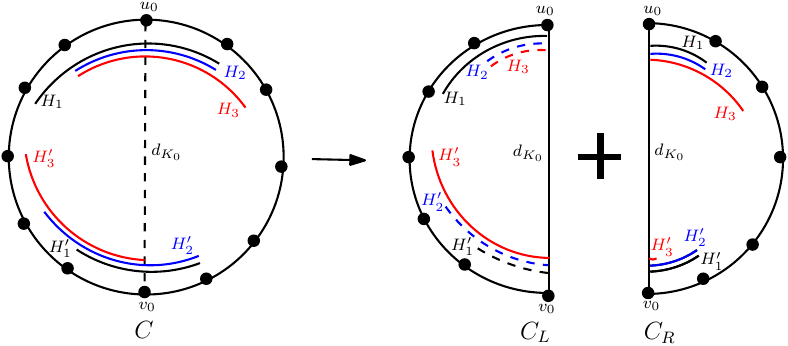}
    \caption{Ordering $\HH_{u_0}$ and $\HH_{v_0}$ subgraphs in $C_L$ and $C_R$. $(H_3\cap C_L)\subseteq (H_2\cap C_L)\subseteq(H_1\cap C_L)$, and $(H_1\cap C_R)\subseteq (H_2\cap C_R)\subseteq(H_3\cap C_R)$. Analogously, $(H'_1\cap C_L)\subseteq (H'_2\cap C_L)\subseteq(H'_3\cap C_L)$, and $(H'_3\cap C_R)\subseteq (H'_2\cap C_R)\subseteq(H'_1\cap C_R)$.}
    \label{fig:leftrightordering}
\end{figure}

\begin{figure}[h!]
    \centering
    \includegraphics[width=1\linewidth]{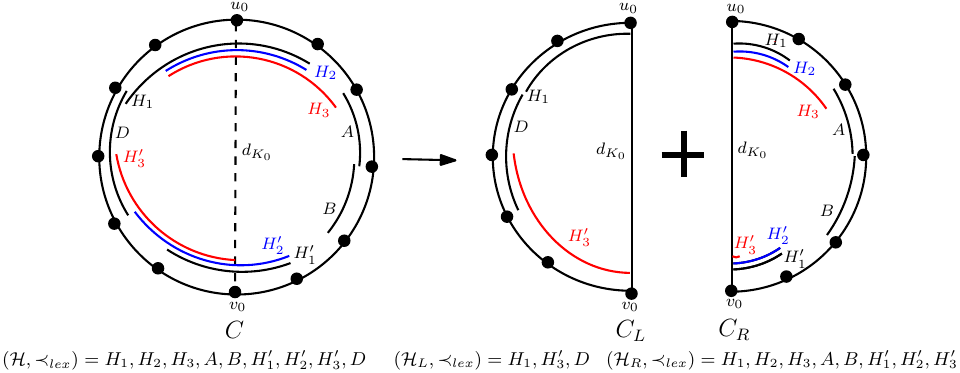}
    \caption{Construction of $\mathcal{H}_L$ and $\mathcal{H}_R$. Any $H\in\mathcal{H}\setminus\{H_1,H'_3\}$ containing $u_0$ or $v_0$ such that $(H\cap C_L)\subseteq (H_1\cap C_L)$ or $(H\cap C_L)\subseteq (H'_3\cap C_L)$ is not included in $\mathcal{H}_L$.}
    \label{fig:constructing H_L and H_R}
\end{figure}

Note that since $\mathcal{H}$
is strict-containment free, so are $\mathcal{H}_L$ and $\mathcal{H}_R$, and that each $H\in\mathcal{H}\setminus\{H_1,H'_s\}$ contributes a run to exactly one of $\mathcal{H}_L$ or $\mathcal{H}_R$, by construction.

By the choice of $K_0$, in the cycle system $(C_R,\mathcal{H}_R,\mathcal{K}_R)$, each $H\in\mathcal{H}_R$ and each $K\in\mathcal{K}_R$ induces a single run on $C_R$.
Hence, by Lemma~\ref{lem:cyclesupport}, there is an outerplanar support $Q_R$ for $(C_R,\mathcal{H}_R,\mathcal{K}_R)$ such that $Q_R=C(\mathcal{H}_R)$, the cycle on $\HH_R$ in lex. cyclic order.

The cycle system $(C_L,\mathcal{H}_L, \mathcal{K}_L)$ is $axax$-free
since $(C,\mathcal{H},\mathcal{K})$ is $axax$-free.
Further, $N(C_L,\mathcal{K}_L)<N$ since $d_{K_0}$ joins two disjoint runs of $K_0$. Hence, by the inductive hypothesis, $(C_L,\mathcal{H}_L,\mathcal{K}_L)$ admits an outerplanar support $Q_L$ with outer cycle $C(\mathcal{H}_L)$.

Since in $Q_L$, the outer cycle is $C(\mathcal{H}_L)$ and
in $Q_R$, the outer cycle is $C(\mathcal{H}_R)$, it follows
that $H_1$ and $H'_s$ are consecutive in the outer cycles of $Q_L$ and $Q_R$.
To obtain a support $Q$ for $(C,\mathcal{H,K})$, we identify
the copy of $H_1$ in $Q_L$ and $Q_R$, and similarly, we identify
the copy of $H'_s$ in $Q_L$ and $Q_R$. It follows that $Q$ is
outerplanar and $C(\HH)$ is the outer cycle in the resulting embedding of $Q$.

Next, we show that $Q$ is a support. 
Consider an arbitrary subgraph $K\in\mathcal{K}$.
We consider three possible cases for the runs of $K$.

Suppose the runs of $K$ lie entirely in $C_R$. Then,
$\mathcal{H}_K$ induces a connected subgraph in $Q$ since every subgraph in $\mathcal{H}$ intersecting $C_R$ has a run in $\mathcal{H}_R$, and that $Q_R$ is a support for $(C_R,\HH_R,\KK_R)$.

Now, suppose the runs of $K$ lie entirely in $C_L$. 
If $\mathcal{H}_K\cap(\mathcal{H}_{u_0}\cup\mathcal{H}_{v_0})=\emptyset$, i.e., if none of the subgraphs in $\HH_K$ contains $u_0$ or $v_0$, then for each subgraph $H\in\HH_K$, we have $H\subseteq \arc(v_0,u_0)$ and thus $H\in\HH_L$ 
since
each $H\in\mathcal{H}$ induces a single run on $C$. 
Since $Q_L$ is a support for $(C_L,\HH_L,\KK_L)$, it follows that
$\mathcal{H}_K$ induces a connected subgraph of $Q$. 
Now, suppose 
$\mathcal{H}_K\cap\mathcal{H}_{u_0}\neq\emptyset$, or $\mathcal{H}_K\cap\mathcal{H}_{v_0}\neq\emptyset$. Assume the former, without loss
of generality. Since the subgraphs in $\mathcal{H}$ were assumed to be strict-containment free,
$K$ intersects a prefix of the sequence of subgraphs in $\mathcal{H}_{u_0}$ in lex. cyclic order, i.e.,
a prefix of the subgraphs $(H_1, \ldots, H_r)$, since
$(H_i\cap C_L)\supseteq (H_{i+1}\cap C_L)$ for $i=1,\ldots, r-1$. 
In particular $H_1\in\mathcal{H}_K$.
Again, as argued, the subgraphs in $\mathcal{H}_R$ are strict-containment free, and hence, they appear consecutively
in lex. cyclic order on the outer cycle of $Q_R$. 
Since $Q_R$ is a support for $(C_R,\mathcal{H}_R,\mathcal{K}_R)$, and $Q_L$ is a support for $(C_L,\mathcal{H}_L,\mathcal{K}_L)$, with $H_1$
appearing in both $Q_L$ and $Q_R$, it follows that $\mathcal{H}_K$ induces a connected subgraph of $Q$.
A similar argument holds when both $\mathcal{H}_K\cap\mathcal{H}_{u_0}\neq\emptyset$ and $\mathcal{H}_K\cap\mathcal{H}_{v_0}\neq\emptyset$.

Finally, suppose that $K$ intersects both $C_L$ and $C_R$.
Since $(C,\mathcal{K})$ is $axax$-free and $K_0$ contains $u_0$ and $v_0$, it follows by Proposition~\ref{prop:chordpartition}, that
$K$ contains either $u_0$ or $v_0$. Therefore,
either $H_1$ or $H'_s$ is in $\mathcal{H}_K$. Now, the
fact that $\mathcal{H}_K$ induces a connected subgraph in $Q$
is identical to the case when $K$ has runs only in $C_L$ s.t. $\mathcal{H}_K\cap\mathcal{H}_{u_0}\neq\emptyset$.
Since $K$ was chosen arbitrarily, $Q$ is a support.
\end{proof}

Now, we are ready to prove the general case.
For the proof of the general case, we require a few technical tools.
For a cycle system $(C,\mathcal{H},\mathcal{K})$, let
$H_0=\arg\min_{H\in\mathcal{H}}\ell(H)$. Let $d_{H_0}=\{u_0,v_0\}$ denote
the chord of $H_0$ attaining the minimum.
The two \emph{derived cycle systems}
$(C,\mathcal{H}_L,\mathcal{K})$ and $(C,\mathcal{H}_R,\mathcal{K})$
corresponding to $H_0$ are defined as follows:

Let $\mathcal{H}_R=\{H\cap \arc[u_0,v_0]: H\in\mathcal{H}\setminus H_0\}\cup H_0'$, where $H_0'$ corresponds to $H_0$, and spans all the vertices of the complementary $\arc[v_0,u_0]$.
We construct $\mathcal{H}_L$ in two steps.
Set $\mathcal{H}'_L=\{H\cap \arc[v_0,u_0]: H\in\mathcal{H}\setminus H_0\}$.
Let $H\in\mathcal{H}$ be s.t. $H\cap \arc(u_0,v_0)\neq\emptyset$.
Since $H_0\cap \arc(u_0, v_0)=\emptyset$ 
with $u_0,v_0\in H_0$, 
and $(C,\mathcal{H})$ is $axax$-free, 
it follows from Proposition~\ref{prop:chordpartition} that in $\arc[v_0,u_0]$, any run of $H$
must be contained in a run of $H_0$. 
Therefore, we remove all such subgraphs from $\HH'_L$.
We set $\mathcal{H}_L=\mathcal{H}'_L\setminus \{H\in\mathcal{H}'_L: H\cap \arc(u_0, v_0)\neq\emptyset\}\cup H_0''$, where $H_0''=H_0\cup \arc[u_0,v_0]$ corresponds to the subgraph $H_0$. That is, 
$H_0''$ extends from $u_0$ and $v_0$ to contain the complementary arc $\arc[u_0, v_0]$.
This defines the \emph{derived cycle systems} 
$(C,\mathcal{H}_L,\mathcal{K})$ and 
$(C,\mathcal{H}_R,\mathcal{K})$ corresponding to the subgraph $H_0$. 
Note that each subgraph $H\in\HH\setminus\{H_0\}$ has a run in exactly one of $\HH_L$ or $\HH_R$.
Moreover, if $\HH$ is containment-free,
and $(C,\HH)$ is $axax$-free, then it follows that $\HH_R$ is strict-containment free.
Figure~\ref{fig:derived} shows a partition of $(C,\mathcal{H},\mathcal{K})$ into the derived systems.

\begin{figure}[ht!]   
    \centering
    \begin{subfigure}{0.32\textwidth}
    \begin{center}
    \includegraphics[scale=.8]{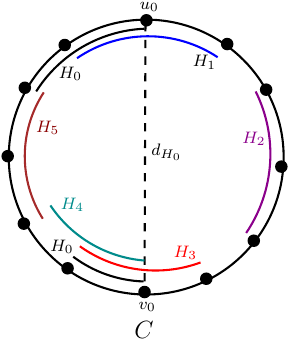}
        \caption{$(C,\HH)$}
            \label{vb1}
    \end{center}
     % \vspace{0.2cm}
    \end{subfigure}
    \begin{subfigure}{0.32\textwidth}
    \begin{center}
    \includegraphics[scale=.8]{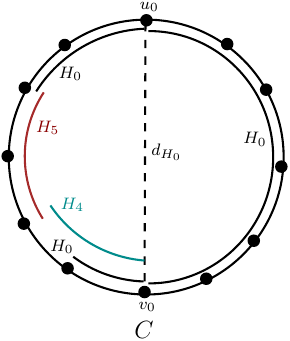}
    \caption{$(C,\HH_L)$}
    \label{fig:vb3}
    \end{center}
     \end{subfigure}
    \begin{subfigure}{0.32\textwidth}
    \begin{center}
    \includegraphics[scale=.8]{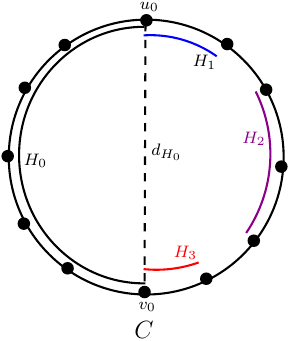}%\vspace{0.85cm}
        \caption{$(C,\HH_R)$}
            \label{vb2}
    \end{center}
     \end{subfigure}
     \caption{Collections $\HH_L$ and $\HH_R$ of derived systems. $\KK$ subgraphs are not shown in the figure.}\label{fig:derived}
     \end{figure}

\begin{proposition}
\label{prop:propp}
Let $(C,\mathcal{H},\mathcal{K})$ be a strong $axax$-free system.
Let $H_0=\arg\min_{H\in\mathcal{H}}\ell(H)$ and that $d_{H_0}=\{u_0,v_0\}$, where $d_{H_0}$ is the chord of $H_0$ attaining 
the minimum. Then, the two derived cycle systems $(C,\mathcal{H}_L,\mathcal{K})$ and
$(C, \mathcal{H}_R, \mathcal{K})$ corresponding
to $H_0$ are also strong $axax$-free.
\end{proposition}
\begin{proof}
Clearly, the cycle systems $(C,\HH_L,\KK)$ and $(C,\HH_R,\KK)$ are $axax$-free.
We show that they satisfy the intersection property.
By the choice of $H_0$, in the derived system $(C,\mathcal{H}_R,\mathcal{K})$, each $H\in\mathcal{H}_R$
induces a single run on $C$. Therefore, it is trivially strong $axax$-free.
Consider the derived system $(C,\HH_L,\KK)$. Let $H\in\mathcal{H}_L$ and $K\in\mathcal{K}$ be such that there are vertices
$h_1, k_1, h_2, k_2$ in cyclic order on $C$ with $h_1,h_2\in H$ and $k_1,k_2\in K$ such that $H\cap K=\emptyset$.
First, suppose that $H=H_0''$. If $K$ has a vertex in $\arc[u_0,v_0]$, then $H\cap K\ne\emptyset$; a contradiction to our assumption.
If $K$ does not have a vertex in $\arc[u_0,v_0]$, then it follows that $H_0$ and $K$ does not satisfy the intersection property in $C$ since $H_0$ and $H_0''$ coincide on $\arc[v_0,u_0]$; again a contradiction to the fact that $(C,\HH,\KK)$ is strong $axax$-free.

So, suppose $H\ne H_0''$.
By construction of $\mathcal{H}_L$, $H\cap \arc(u_0, v_0)=\emptyset$.
This implies $H$ and $K$ do not satisfy the intersection property
in $(C,\mathcal{H},\mathcal{K})$, contradicting the assumption that
it is strong $axax$-free.
\end{proof}

Now, we prove the result for the general setting i.e., when subgraphs in $\HH\cup\KK$ induce more than one run
on $C$.

\outintsupport*
\begin{proof}
By Lemma~\ref{lem:contfree}, we assume wlog that $\HH$ is containment-free.
We prove the result by induction on $N=N(C,\mathcal{H})=\sum_{H\in\mathcal{H}}(n_H-1)\ge 0$.
If $N=0$, then each $H\in\mathcal{H}$ induces a single run on $C$, and by
Lemma~\ref{lem:basecase}, we are done.
So, we can assume that $N\ge 1$ i.e., $n_H\ge 2$ for some $H\in\mathcal{H}$.
Suppose the theorem holds for all strong $axax$-free systems $(C,\mathcal{H}',\mathcal{K})$ with $N(C,\mathcal{H}')<N$.

Let $(C,\mathcal{H,K})$ be a strong $axax$-free system with $N(C,\mathcal{H})=N$.
Let $H_0=\arg\min_{H\in\mathcal{H}}\ell(d_H)$ with $d_{H_0}=\{u_0,v_0\}$, and such that $H_0\cap \arc(u_0,v_0)=\emptyset$. 

Consider the derived cycle systems $(C,\mathcal{H}_L,\mathcal{K})$
and $(C,\mathcal{H}_R,\mathcal{K})$. By Proposition~\ref{prop:propp},
both derived cycle systems are strong $axax$-free.

In the cycle system $(C,\mathcal{H}_R, \mathcal{K})$, each
subgraph in $\mathcal{H}_R$ induces a single run and $\HH_R$ is strict-containment free. By Lemma~\ref{lem:basecase}, there is an outerplanar support $Q_R$
whose outer cycle is $C(\mathcal{H}_R)$.
By the choice of $H_0$, $N(C,\mathcal{H}_L)<N$, and by Proposition~\ref{prop:propp}, $(C,\HH_L,\KK)$ is a strong $axax$-free system. By the inductive hypothesis therefore, $(C,\HH_L,\KK)$
admits an outerplanar support $Q_L$.
By construction of $\mathcal{H}_L$ and $\HH_R$, each $H\in\HH\setminus\{H_0\}$ has its representative vertices in exactly one of $Q_L$ or $Q_R$.
Let $Q$ be the graph obtained by identifying $H''_0$ in $Q_L$ and $H'_0$ in
$Q_R$. The graph $Q$ is clearly outerplanar.
It remains to show that $Q$ is an intersection support for $(C,\HH,\KK)$.

Let $K\in\KK$ be arbitrary.
We show that $\HH_K$ induces a connected subgraph of $Q$.
We consider three cases for the runs of $K$.
First, suppose that all runs of $K$ lie in $\arc[u_0,v_0]$. Then all subgraphs in $\HH$ intersecting $K$ have a representative in $\HH_R$.
The fact that $\HH_K$ is connected in $Q$ follows from the fact that $Q_R$ is an intersection support for 
$(C,\HH_R,\KK)$.

Now, suppose that all runs of $K$ lie entirely in $\arc[v_0,u_0]$.
If each $H\in\HH_K$ has its representative in $\HH_L$, then the fact that $\HH_K$ induces a connected subgraph of $Q$ follows from the fact that $Q_L$ is an intersection support for $(C,\HH_L,\KK)$.
Otherwise, if an $H\in\HH_K$ does not have its representative in $\HH_L$, then by construction of $\HH_L$, $H\cap\arc(u_0,v_0)\ne\emptyset$.
By Proposition \ref{prop:chordpartition} therefore, $H$ contains $u_0$ or $v_0$ and each run of $H$ in $\arc[v_0,u_0]$ is contained in a run of $H_0$.
It follows that $K\cap H_0\ne\emptyset$ and thus $K\cap H''_0\ne\emptyset$.
Recall that the subgraph in $\HH$ containing $u_0$ or those containing $v_0$ appear consecutively in the outer cycle $C(\HH_R)$ of $Q_R$, and $K$ intersects a prefix of these sequences of subgraphs.
Since $K$ intersects $H_0$ which
contains both $u_0$ and $v_0$, $\HH_K$ induces a connected subgraph of $Q$.

Finally, let $K$ intersect both $\arc[u_0,v_0]$ and $\arc[v_0,u_0]$.
Note that in this case $K$ intersects $H_0$ (and hence with $H'_0$ and $H''_0$) since $u_0,v_0\in H_0$ and $(C,\HH,\KK)$ is strong $axax$-free.
Each $H\in\HH\setminus\{H_0\}$ has a representative vertex in one of $Q_L$ or $Q_R$ which are supports for the derived cycle systems with $H_0$ in both $Q_L$ and $Q_R$.
It follows that $\HH_K$ induces a connected subgraph of $Q$.
Since $K$ was chosen arbitrarily, $Q$ is a support.
\end{proof}

\subsection{Implementation}\label{sec:implementation}
In this section, we show a polynomial running time of our algorithm to construct an outerplanar support.
First, we give a short
argument that $|\mathcal{H}|$ and $|\mathcal{K}|$ are polynomially bounded.
To show this, we again use an argument based on VC-dimension, together with results of~\cite{ackerman2020coloring} and~\cite{RR18}.
Possibly the bound here is not sharp, but it is sufficient for our purposes since we only want to establish a polynomial running time.
We start with the definition of an $abab$-free cycle system.

\begin{definition}[$abab$-free]
\label{defn:abab}
Let $(C,\mathcal{H})$ be a cycle system.
$H,H'\in\mathcal{H}$ are an $abab$-pair if there are four distinct vertices 
$a_1,b_1,a_2,b_2$ in cyclic order on $C$ such that $a_1,a_2\in H\setminus H'$ and $b_1,b_2\in H'\setminus H$.
$(C,\mathcal{H})$ is $abab$-free if there are no $abab$-pairs in $\mathcal{H}$.
\end{definition}

Note that, if $(C,\HH)$ is $axax$-free, then it is $abab$-free.
The notion of cycle $abab$-free is equivalent to that of $ABAB$-free hypergraphs introduced by Ackerman et al. \cite{ackerman2020coloring}.
For the ease of exposition, we use $abab$-free instead of $ABAB$-free in the proof of Theorem \ref{thm:vcdim} below.

\begin{theorem}
\label{thm:vcdim}
Let $(G,\mathcal{H},\mathcal{K})$ be an outerplanar non-piercing graph system.
Then, $|\mathcal{H}|,|\mathcal{K}|=O(n^4)$, where $n=|V(G)|$.
\end{theorem}
\begin{proof}
Let $(C,\HH,\KK)$ be the resulting cycle system.
By Lemma \ref{lem:npaxax-free}, $(C,\mathcal{H},\mathcal{K})$
is strong $axax$-free, and this implies that $(C,\mathcal{H})$
and $(C,\mathcal{K})$ are $abab$-free.

Ackerman et al. \cite{ackerman2020coloring} showed
that $abab$-free hypergraphs are exactly those that can be
represented by a set of \emph{pseudodisks}\footnote{A family $\mathcal{D}$ of compact regions in the plane s.t. the boundary of each region $D\in\mathcal{D}$ is a simple Jordan curve and for any two regions $D, D'$ their boundaries intersect either $0$ or 2 times
is called a set of pseudodisks.} that contain a common point. 
Hence, the hypergraphs defined by $\mathcal{H}$
can be represented s.t. the vertices of $C$ are points in the plane and
the hyperedges in $\mathcal{H}$ are pseudodisks containing the
origin, and each hyperedge is defined by the set of points of $C$
in the pseudodisk.

We show that the VC-dimension of the set system $(C,\HH)$ is at most 4.
The results of Raman and Ray \cite{RR18} imply that for any set
$P$ of points and any set $\mathcal{D}$ of pseudodisks, there
is a planar support $Q$ on $P$, i.e., a planar graph on $P$ s.t.
the points in $D$ induce a connected subgraph of $Q$ for each
$D\in\mathcal{D}$.
If a set of 5 points can be shattered, then for each pair of points,
there is a hyperedge containing that pair. But, this implies that support for these 5 points is $K_5$, contradicting the fact that it is planar.

Hence, the VC-dimension of the set system $(C,\mathcal{H})$ is
at most 4, and this implies, by the Sauer-Shelah lemma (See Chapter 10 in~\cite{Mat02LecDiscGeom})
that $|\mathcal{H}|=O(n^4)$.
Similarly, $|\mathcal{K}|=O(n^4)$.
\end{proof}

\begin{theorem}\label{thm:runningtime-op}
If $(G,\HH,\KK)$ is an outerplanar non-piercing intersection system, then an intersection support can be computed in time $O(n^6)$ where $n=|V(G)|$.
\end{theorem}
\begin{proof}
Let $(C,\HH,\KK)$ be the resulting strong $axax$-free cycle system.
Then $|C|=n$, and by Theorem~\ref{thm:vcdim}, $|\HH|,|\KK|=O(n^4)$.
We only need to show that our algorithm in Theorem~\ref{thm:outintsupport} runs in time $O(n^6)$.

If each $X\in\HH\cup\KK$ induces a single run on $C$, then by Lemma~\ref{lem:cyclesupport}, $C(\HH)$ is the desired support.
To construct $C(\HH)$, we walk along $C$ and at each vertex $v$, add the subgraphs $H\in\HH$ to $C(\HH)$
s.t. $s(H)=v$.
For each vertex, we order the subgraphs in increasing order of $t(H)$, which can be computed by
sorting the subgraphs $H$ with $s(H)=v$.
Thus, $C(\HH)$ can be computed in time $O(|C| + |\mathcal{H}|\log|\mathcal{H}|)=O(n^4\log{n})$.

For the case when $\HH$ consists of single runs and a subgraph in $\KK$ can have multiple runs, finding a chord $d$ of smallest length can be done by storing
the subgraphs in $\mathcal{K}$ in a heap ordered by $\ell(K)$.
The time taken to partition the problem into two sub-problems is
$O(|C|\max\{|\mathcal{H}|,|\mathcal{K}|\})$
as we need to go through each subgraph in $\mathcal{H}$ and $\mathcal{K}$, and split the runs into the two sub-problems. 
We add at most $|C|-3$ chords in $C$ since the resulting support is outerplanar.
Hence, the
total running time in this case is $O(|C|^2|\max\{|\mathcal{H}|,|\mathcal{K}|\})
= O(n^6)$.

In the general case, a subgraph in $\HH$ can have multiple runs.
To compute an $H\in\mathcal{H}$
minimizing $\ell(H)$, we can store the subgraphs in $\mathcal{H}$
in a heap ordered by $\ell(H)$. We can split the problem into
two sub-problems in $O(|C|\max\{|\mathcal{H}|,|\mathcal{K}|\})$
time. Since we add at most $|C|-3$ chords, the overall running time
is bounded above by $O(|C|^2\max\{|\mathcal{H}|,|\mathcal{K}|\})$.
Hence, the overall running time of the algorithm is $O(n^6)$.
\end{proof}

\subsection{Primal and Dual Supports: Revisited}
\label{sec:ababfree}

The existence of an intersection support also implies a primal and a dual support.
However, we emphasize here that for a primal and a dual support, we require strictly weaker conditions.
In particular, for a primal support, we require the cycle system $(C,\HH)$ to be $abab$-free and not necessarily $axax$-free, and for a dual support, we require $(C,\HH)$ to be $axax$-free.
The result for a primal support directly follows from the following lemma of Raman and Singh~\cite{raman2025supportsgraphsboundedgenus}.

\begin{lemma}[\cite{raman2025supportsgraphsboundedgenus}]
\label{lem:nblock}
Let $(C,\HH)$ be an $abab$-free cycle system.
Then, we can add a set $D$ of non-intersecting chords in $C$ such that each $H\in\mathcal{H}$ 
induces a connected subgraph of $C\cup D$. Further, the set $D$ of non-intersecting chords to add can be computed in time $O(mn^4)$ where $n=|C|$, and $m=|\HH|$.
\end{lemma}

\begin{theorem}\label{thm:primalopsupport}
Let $(C,\HH)$ be an $abab$-free cycle system and $c:V(C)\to\{\R,\B\}$ be a 2-coloring of the vertices of $C$.
Then there is a primal support for $(C,\HH)$ that is an outerplanar graph.
\end{theorem}
\begin{proof}
For each $v\in V(C)$ such that $c(v)=\R$, we delete it and make its neighbours adjacent.
Let $C'$ be the resulting cycle obtained form $C$, and $\HH'=\{H\cap C':H\in\HH\}$.
Then $(C',\HH')$ is an $abab$-free cycle system.
By Lemma \ref{lem:nblock}, there is a set $D$ of non-intersecting chord added to $C'$ such that each subgraph in $\HH'$ induces a connected subgraph of $C'\cup D$. Thus $C'\cup D$ is an outerplanar primal support for $(C',\HH')$.
Note that this is also a primal support for $(C,\HH)$.
\end{proof}

For the dual setting, one may likewise hope that
the $abab$-free condition is sufficient to obtain a
dual support. Unfortunately, that is not the case as the following
example shows:
Let $G$ be the graph on $\{1,2,\ldots 6\}$ as shown in Figure~\ref{fig:asteroidal}.
Let $\mathcal{H}$ be subgraphs 
$H_1=\{1,2,3\}$, $H_2=\{3,4,5\}$, $H_3=\{5,6,1\}$ and $H_4=\{2,4,6\}$.
It can be checked that $\mathcal{H}$ is $abab$-free.
The dual support for $(G,\HH)$ is $K_4$, a complete graph on 4 vertices, which is not outerplanar.

The reason why $(G,\mathcal{H})$ does not admit a dual outerplanar
support is that the induced system $(C,\mathcal{H})$ is not
$axax$-free - $H_4$ forms an $axax$-pair with each of $H_1, H_2$
and $H_3$.
However, $axax$-freeness is a sufficient condition for $(C,\HH)$ to exhibit an outerplanar dual support and it follows from Theorem \ref{thm:outintsupport}.

\begin{theorem}\label{opdualsupport}
If $(C,\HH)$ is $axax$-free, then it admits an outerplanar dual support.
\end{theorem}
\begin{proof}
Let $\KK=\{\{v\}:v\in V(C)\}$ be a collection of subgraphs each consisting of a single vertex.
Then $(C,\KK)$ is $axax$-free.
Therefore, the intersection system $(C,\HH,\KK)$ is $axax$-free.
Note that the strong $axax$-free property is trivially satisfied by $(C,\HH,\KK)$.
By Theorem \ref{thm:outintsupport}, it has an intersection support for $Q$, which is also a dual support for $(C,\HH)$ since each $v\in V(C)$ corresponds to a subgraph in $K\in\KK$.
\end{proof}
\begin{figure}[ht!]
    \centering
    \includegraphics[width=0.3\linewidth]{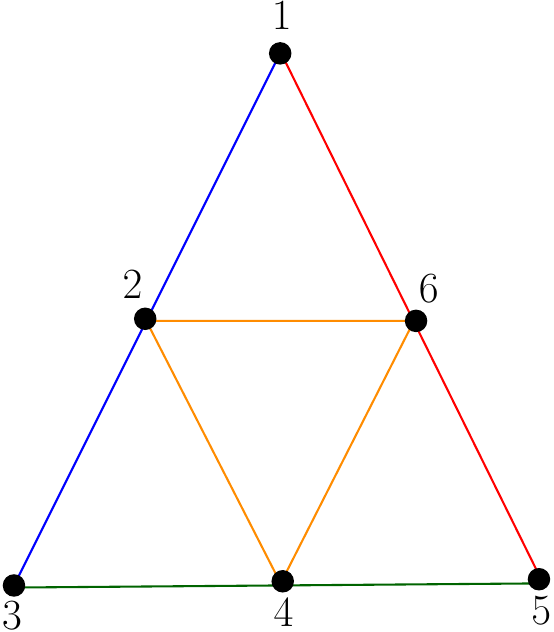}
    \caption{$abab$-free cycle system whose only dual support is $K_4$. The induced subgraphs are shown by four different colors on edges.}
    \label{fig:asteroidal}
\end{figure}

Again, we might hope that for an outerplanar intersection system
$(G,\mathcal{H},\mathcal{K})$, if it satisfies the $axax$-free property,
then like the primal and dual, we can obtain an outerplanar intersection support.
However, the following example shows that in order to construct 
an intersection support, one needs the strong $axax$-free property:
Let $C=(1,2,\ldots 7,1)$ be a cycle as shown in Figure~\ref{fig:alternatesequence}.
Let $\mathcal{H}$ consist of induced subgraphs $H_1=\{1,5\},H_2=\{2\},H_3=\{3,4\}$ and $H_4=\{6,7\}$, and $\mathcal{K}$ consist of induced subgraphs $K_1=\{1,2\},K_2=\{2,3\},K_3=\{4,6\},K_4=\{1,7\},K_5=\{2,7\}$ and $K_6=\{4,5\}$ of $C$.
It is easy to check that both $(C,\mathcal{H})$ and $(C,\mathcal{K})$ are $axax$-free. However, for each $K_i$, $i=1,\ldots, 6$, there is exactly
one pair $H_{k},H_\ell$, $k\neq\ell\in\{1,\ldots, 4\}$ s.t. $H_k$ and $H_\ell$ intersect $K_i$, and hence, the intersection support is a complete graph on $4$ vertices
(see Figure~\ref{fig:k4support}), which is not outerplanar.

Observe that vertices $1,4,5,6$ are in cyclic sequence with $1,5\in H_1$ and $4,6\in K_3$ such that $H_1\cap K_3=\emptyset$. Similarly, 
vertices $1,2,5,7$ appear in cyclic sequence with $1,5\in H_1$ and $2,7\in K_5$ such that $H_1\cap K_5=\emptyset$. Hence, the graph system $(C,\HH,\KK)$ does not satisfy the strong $axax$-free property.

\begin{figure}[ht!]
\centering
\begin{subfigure}{0.45\textwidth}
    \begin{center}
    \includegraphics[scale=.5]{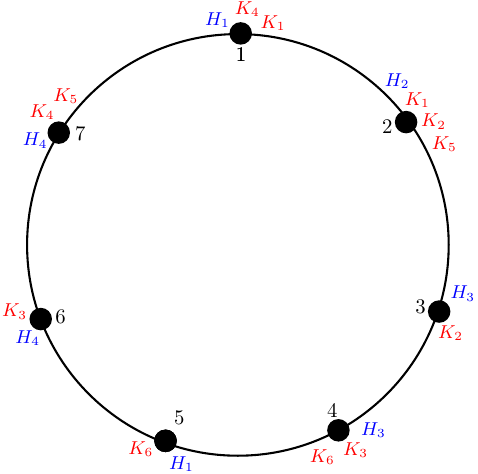}
    \caption{Alternate sequences of $H_1$ with $K_3$ and $K_5$.}
    \label{fig:alternatesequence}
    \end{center}
\end{subfigure}
\hfill
\begin{subfigure}{0.45\textwidth}
\begin{center}
    \includegraphics[scale=.5]{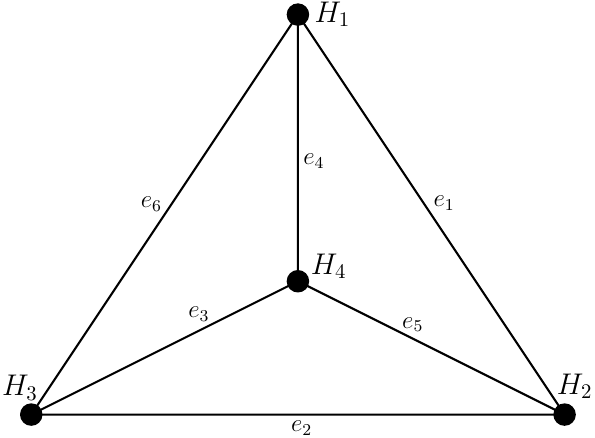}
    % \vspace{0.85cm}
    \caption{Non-outerplanar support. Edge $e_i$ corresponds to $K_i$ for $i=1\ldots 6$.}
\label{fig:k4support}
\end{center}
 \end{subfigure}
 \vspace{0.5cm}
 \caption{An example of an $axax$-free, but not strong $axax$-free cycle system $(C,\mathcal{H,K})$ (shown in left), that does not admit an outerplanar support (shown in right).}
 \label{fig:nonopsupport}
 \end{figure}

\section{Conclusion}\label{sec:conclusion}
In this paper, we gave algorithms to construct support for hypergraphs defined on a host graph $G$ and collections of non-piercing subgraphs of $G$.
When $G$ is outerplanar, we showed that there are outerplanar primal, dual, and intersection supports.
Deciding if an abstract hypergraph admits a 2-outerplanar support was shown to be NP-hard by Buchin et al. \cite{buchin2011planar}.
The complexity of the decision problem is not known for outerplanar support.
It is plausible that the machinery we used for the construction of outerplanar support may help in resolving this open problem.

For the case when $G$ has treewidth $t$, we constructed primal, dual and intersection supports of treewidth at $O(2^t)$, $O(2^{4t})$ and $2^{O(2^t)}$, respectively.
All our algorithms run in polynomial time in the number of vertices of $G$ if $t$ is bounded above by a constant.
We also construct examples where the exponential blow-up in the treewidth of any primal or dual supports can not be improved.
In the case of dual setting, we constructed a support of treewidth $O(2^{4t})$.
We leave it an open question if the exponent $4t$ can be improved to $t$.
Similarly, in the intersection setting, there is a double exponential in the treewidth of an intersection support.
We believe that it is possible to get an intersection support of treewidth $O(2^{ct})$ for some constant $c$.

From our results, it follows that a non-piercing outerplanar graph system admits an outerplanar support. Since $\tw(G)=2$ for an outerplanar graph $G$, we wonder if for a non-piercing graph system of treewidth 2, there exists a primal/dual or intersection support of treewidth 2.

\bibliography{ref}
\end{document}